\documentclass{amsart}

\usepackage[utf8]{inputenc}
\usepackage{tikz}
\usepackage{amsfonts, amssymb, amsmath, amsthm}
\usepackage{graphicx}
\usepackage{float, comment}
\usepackage[hidelinks]{hyperref}
\usepackage{verbatim}
\usepackage{mathtools, nccmath}

\usepackage{enumitem}
\usepackage{verbatim}

\newtheorem{teor}{Theorem}[section]
\newtheorem{lema}[teor]{Lemma}
\newtheorem{prop}[teor]{Proposition}
\newtheorem{fact}[teor]{Fact}
\newtheorem*{teorsin}{Theorem}
\newtheorem{defin}[teor]{Definition}
\newtheorem*{claim}{Claim}
\newtheorem*{prop*}{Proposition}
\newtheorem{cor}[teor]{Corollary}

\newtheorem{problem}[teor]{Problem}
\newtheorem{question}[teor]{Question}
\newtheorem{external claim}[teor]{Claim}

\newcommand{\bigslant}[2]{{\raisebox{.2em}{$#1$}\left/\raisebox{-.2em}{$#2$}\right.}}
\newcommand{\bslant}[2]{{\raisebox{.2em}{$#1$}/\raisebox{-.2em}{$#2$}}}

\newcommand{\C}{{\mathfrak C}}
\newcommand{\R}{\mathbb{R}}

\DeclareMathOperator{\tp}{{tp}}
\DeclareMathOperator{\qftp}{{tp^{qf}}}
\DeclareMathOperator{\Th}{{Th}}
\DeclareMathOperator{\aut}{{Aut}}
\DeclareMathOperator\dcl{dcl}
\DeclareMathOperator\auto{Aut}
\DeclareMathOperator\stab{Stab}
\DeclareMathOperator\dom{dom}

\newcommand{\orcidlogo}{\includegraphics[height=\fontcharht\font`\B]{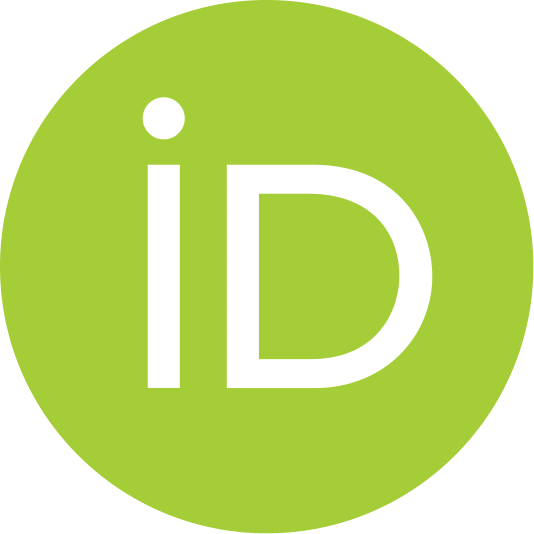}}
\newcommand{\orcid}[1]{\href{#1}{\orcidlogo #1}}

\usepackage[
backend=bibtex,
style=alphabetic,
]{biblatex}
\usepackage{stackengine}
\stackMath
\DeclareMathOperator*\dcap{{\stackinset{r}{-0.35ex}{c}{-1.9pt}{\downarrow}
		{\bigcap}\mkern2mu}}

\title{Maximal stable quotients of invariant types in NIP theories}
\author{Krzysztof Krupi\'{n}ski}
\author{Adri\'{a}n Portillo}

\thanks{\noindent Both authors are supported by the Narodowe Centrum Nauki grant no. 2016/22/E/ST1/00450. The first author is also supported by  the Narodowe Centrum Nauki grant no. 2018/31/B/ST1/00357.}

\address{Instytut Matematyczny Uniwersytetu Wroc{\l}awskiego, pl. Grunwaldzki 2, 50-384 Wroc{\l}aw, Poland}
\address{Krzysztof Krupi\'{n}ski \orcid{https://orcid.org/0000-0002-2243-4411}
}
\email{Krzysztof.Krupinski@math.uni.wroc.pl}
\address{Adri\'{a}n Portillo \orcid{https://orcid.org/0000-0001-9354-8574}}
\email{Adrian.Portillo-Fernandez@math.uni.wroc.pl}

\keywords{Stable quotient, hyperimaginary, invariant type.}
\subjclass[2020]{03C45}

\begin{document}
	
	\begin{abstract}
		For a NIP theory $T$, a sufficiently saturated model $\C$ of $T$, and an invariant (over some small subset of $\C$) global type $p$, we prove that there exists a finest relatively type-definable over a small set of parameters from $\C$ equivalence relation on the set of realizations of $p$ which has stable quotient. This is a counterpart for equivalence relations of the main result of \cite{MR3796277} on the existence of maximal stable quotients of type-definable groups in NIP theories. Our proof adapts the ideas of the proof of this result, working with relatively type-definable subsets of the group of automorphisms of the monster model as defined in \cite{hrushovski2021order}.
	\end{abstract}
\maketitle
	\section{Introduction}	
	Stability theory, developed in the 1970s and 1980s, is a core part of model theory. 
One of the main goals of modern 
model theory is to extend various ideas and results of stability theory to appropriate unstable contexts; it is 
particularly important to distinguish interesting contexts extending stability. Two of the approaches are to either impose some general global assumptions on the theory (e.g., NIP, simplicity, $\textrm{NSOP}_1$) or some local ones (such as working with a stable definable set or generically stable type) in order to prove some structural results. Another ubiquitous strategy is to look at hyperdefinable sets 
(i.e. quotients by type-definable equivalence relations) and assume (or prove) their good properties (e.g., boundedness) to obtain further results. 

	

	Bounded quotients have been studied thoroughly and played an important role in model theory and its applications for many years (e.g., to approximate subgroups). However, stable quotients have not been studied so deeply. 
The project originates with a talk by Anand Pillay in Lyon in 2009 on finest stable hyperdefinable quotients in NIP theories. The case of type-definable groups was worked out in \cite{MR3796277}; some questions from \cite{MR3796277} were answered in \cite{10.1215/00294527-2022-0023}. The general case of first order theories will be 
studied in this paper.
%
One should mention that it is well known that hyperimaginaries can be treated as imaginaries in continuous logic, and stability of hyperdefinable sets is equivalent to stability (of imaginary sorts) in the sense of continuous logic. But we will not be using this 
approach in the present paper.
	

Let $T$ be a complete theory, $\C \models T$ a monster model (i.e., $\kappa$-saturated and strongly $\kappa$-homogeneous for a strong limit cardinal 
$\kappa> |T|$) in which we are working, and $A \subseteq \C$ a {\em small} set of parameters (i.e., $|A| < \kappa$); a cardinal $\gamma$ is {\em bounded} if  $\gamma < \kappa$. 
Let $X/E$ be a {\em hyperdefinable set over $A$}, i.e. $X$ is an $A$-type-definable set and $E$ an $A$-type-definable equivalence relation on $X$. The {\em complete type over $A$} of an element of $X/E$ can be defined as the $\aut(\C/A)$-orbit of that element, or the preimage of this orbit under the quotient map, or the partial type defining this preimage.
	
	\begin{defin}
		A hyperdefinable 
		over $A$ set $X/E$ is \emph{stable} if for every $A$-indiscernible sequence $(a_i,b_i)_{i<\omega}$ with $a_i\in X/E$ for all (equivalently, some) $i<\omega$, we have  
		$$\tp(a_i,b_j/A) = \tp(a_j,b_i/A)$$ for all (some) $i\neq j < \omega$.
	\end{defin}

Let $G$ be a $\emptyset$-$\bigwedge$-definable group. Since stability of hyperdefinable sets 
is invariant under type-definable bijections and closed under taking products and type-definable subsets (see \cite[Remark 1.4]{MR3796277}), it is clear that there always exists a smallest $A$-type-definable subgroup $G^{st}_A$ such that the quotient $G/G^{st}_A$ is stable. The main result of \cite{MR3796277} says that under NIP, $G^{st}_A$ does not depend on $A$, and so it is the smallest type-definable (over parameters) subgroup with stable quotient $G/G^{st}$. Moreover, it is  $\emptyset$-type-definable and normal.

In \cite{10.1215/00294527-2022-0023}, we gave several characterizations of stability 
for hyperdefinable sets both with and without assuming that the theory $T$ has NIP (see \cite[Section 2]{10.1215/00294527-2022-0023}). We also proved that in distal theories all stable hyperdefinable sets are bounded, 
i.e. of bounded size (\cite[Corollary 3.5]{10.1215/00294527-2022-0023}). 
In Proposition \ref{weak stability iff stability}, we will deduce another characterization of stability under NIP, via the so-called weak stability, generalizing \cite[Proposition 4.2]{https://doi.org/10.1002/malq.200610046} from the definable to the hyperdefinable context.

	When we lack the group structure, a natural counterpart of taking the quotient by a subgroup is to take 
	 the quotient by an equivalence relation. 
	Thus, it is natural to ask if similar results to the ones appearing in \cite{MR3796277} hold outside of the context of type-definable groups. 
However, the naive counterpart of \cite[Theorem 1.1]{MR3796277} is easily seen to be false. 
Namely, in general, for any non-stable type-definable set $X$ (e.g. the home sort of a non-stable theory), a finest type-definable (over an arbitrary small set of parameters)  equivalence relation on $X$ with stable quotient does not exist. The reason is that given any type-definable equivalence relation $E$ on $X$ with stable quotient, $E$ is not the relation of  equality, so we can find an $E$-class which contains at least two distinct elements $a$ and $b$. Then, the equivalence relation on $X$ being the intersection of $E$ and the relation $\equiv_a$ of having the same type over $a$ is strictly finer that $E$ and has stable quotient by \cite[Remark 1.4]{MR3796277} (as both $X/E$ and $X/\!\!\equiv_a$ are stable).	

	Let $\C\prec \C'$ be two monster models of a NIP theory $T$ 
	such that $\C$ is small in $\C'$.
	Recall that a {\em relatively type-definable over a (small) set of parameters $B$} subset of a set $Y$ is the intersection of $Y$ with a set which is type-definable over $B$. 
	The main result of this paper is the following theorem which will be proved in Section \ref{section: main theorem}.

	\begin{teorsin}
		Assume NIP. Let $p(x)\in S(\mathfrak{C})$ be an $A$-invariant type. Assume that $\C$ is at least $\beth_{(\beth_{2}(\lvert x \rvert+\lvert T \rvert + \lvert A\rvert))^+}$-saturated. Then, there exists a finest equivalence relation $E^{st}$ on $p(\C')$ 
		relatively type-definable over a small 
		(relative to $\C$) set of parameters of $\C$ and with stable quotient $p(\C')/E^{st}$.
	\end{teorsin}

	Our proof is via a non-trivial adaptation of the ideas from the proof of the main theorem of \cite{MR3796277}, using relatively type-definable subsets of the group of automorphisms of the monster model (as defined in \cite{hrushovski2021order}).

	We do not know whether $E^{st}$ is relatively type-definable over $A$. At the end of Section \ref{section: main theorem}, we will observe that if it was true, then the specific (large) saturation degree assumption  in the above theorem could be removed. Another question is whether one could drop the invariance of $p$ hypothesis from the above theorem. If such a strengthening is true, a proof would probably require some new tricks.

In Section \ref{section: basics}, we prove several basic results concerning the existence of finest relatively type-definable equivalence relations with stable quotients some of which are used in Section \ref{section: main theorem}, and we discuss the transfer of the existence of finest relatively type-definable equivalence relations with stable quotients between models. 
	
	In Section \ref{section: main theorem}, we prove the main theorem of this paper stated above. 

	In the last section, we compute $E^{st}$ in two concrete examples which are expansions of local orders. 
	In fact, in these examples, we give full classifications of all relatively type-definable 
	 over a small subset of $\C$ equivalence relations on $p(\C')$ for a suitable invariant type $p \in S(\C)$.\\

	
We finish the introduction presenting the framework of this paper.	
	Let $T$ be a complete first-order theory 
	of infinite models in a language $L$. Let $\C\prec\C'$ be models of  $T$ such that $\C$ is $\kappa$-saturated 
with a strong limit cardinal $\kappa>\lvert T\rvert$,  and $\C'$ is $\kappa'$-saturated and strongly $\kappa'$-homogeneous with a strong limit cardinal $\kappa'>\lvert \C\rvert$. We say that $\kappa$ is the {\em degree of saturation of $\C$} and $\kappa'$ is the {\em degree of saturation of $\C'$}. We say that a set is {\em $\C$-small} if its cardinality is smaller than $\kappa$ and {\em $\C'$-small} if its cardinality is smaller than $\kappa'$.
	Note that $|T|$ is the cardinality of the set of all formulas in $L$.
	Unless stated otherwise, $p(x)$ will always be a type in $S_x(\C)$ invariant over some $\C$-small $A\subseteq \C$, where $x$ is a $\C$-small tuple of variables. (In fact, instead of assuming that $\kappa$ is a strong limit cardinal, in Section 2 it is enough to assume that $\kappa >2^{|T|+|A|}$ and in Section 3 that $\kappa \geq \beth_{(2^{2^{|T| + |A|} +\lvert x \rvert})^+}$.) Whenever $B \subseteq \C'$, by $p \!\upharpoonright_B$ we mean the restriction to $B$ of the unique extension of $p$ to an $A$-invariant type in $S(\C')$. If $E$ is a type-definable equivalence relation and $a$ is an element of its domain, $[a]_E$ denotes the $E$-class of $a$.
	

	\section{Basic results and transfers between models}\label{section: basics}

	The goal of this section is to present a useful criterion that allows us to check whether a relatively type-definable over a $\C$-small $B\subseteq \C$ equivalence  relation $E$  on $p(\C')$ with stable quotient is, in fact, the finest one (see Lemma \ref{equivalence for being the finest}). As a corollary, we get the transfer to elementary extensions of $\C$ of the property of being the finest relatively type-definable equivalence relation on $p(\C')$ (see Corollary \ref{corollary: from C to C_1}). We also take the opportunity to prove  a new characterization of stability of hyperdefinable sets in NIP theories (see Proposition \ref{weak stability iff stability}). 
	
	Let $E$ be a type-definable equivalence relation on a type-definable subset $X$ of $\C^\lambda$, where $\lambda< \kappa$. 
	The following definition is the hyperimaginary analogous of \cite[Definition 1.2]{https://doi.org/10.1002/malq.200610046}.
	\begin{defin}
		A hyperdefinable (over $A$) set $X/E$ is \emph{weakly stable} if for every $A$-indiscernible sequence $(a_i,b_i,c)_{i<\omega}$ with $a_i,b_i\in X/E$ for all (equivalently, some) $i<\omega$, we have  
		$$\tp(a_i,b_j,c/A) = \tp(a_j,b_i,c/A)$$ for all (some) $i\neq j < \omega$.
	\end{defin}

We obtain a hyperdefinable counterpart of \cite[Proposition 4.2]{https://doi.org/10.1002/malq.200610046}.

\begin{prop}\label{weak stability iff stability}
		A hyperdefinable set $X/E$ which has NIP is weakly stable if and only if it is stable.
\end{prop}

\begin{proof}
	Without loss of generality, assume that both $X$ and $E$ are type-definable over the empty set.
	
	It is clear that stable sets are weakly stable, even without the $NIP$ assumption. By \cite[Theorem 2.10]{10.1215/00294527-2022-0023}, under the $NIP$ assumption, the stability of $X/E$ is equivalent to the fact that every indiscernible sequence of elements of $X/E$ is totally indiscernible. 
	Hence, it is enough to show that weak stability of $X/E$ also implies this property.
		
		Suppose that the sequence $(a_i)_{i<\omega}$ in $X/E$ is indiscernible but not totally indiscernible. Let us, without loss of generality, replace $\omega$ by $\mathbb{Q}$. Then, there exist a natural number $n$ and 
		 $j<n-1$ such that $$\tp(a_{j},a_{{j+1}}/A)\neq \tp(a_{{j+1}},a_{j}/A),$$ where $A$ is the set of all $a_{k}$ for $k<n$ distinct from $j$ and $j+1$. Choose any rationals $l_0<l_1<\dots$ in the interval $(j,{j+1})$. Let $b_i:=a_{l_i}$ for $i<\omega$. Then, the sequence $(b_i)_{i<\omega}$ is $A$-indiscernible and $\tp(b_i,b_j/A)\neq \tp(b_j,b_i/A)$ for all $i<j<\omega$.
Let $a$ be an enumeration of $A$.
We conclude that the sequence $(b_i,b_i,a)_{i<\omega}$ contradicts the weak stability of $X/E$.
\end{proof}
 	
	Next, we present a definition that we use throughout the whole section. This definition first appeared in \cite[Definition 3.2]{doi:10.1142/S0219061319500120}.
	\begin{defin}
		Let $A\subseteq \mathcal{M}\subseteq B$ and $q(x)\in S(B)$.
 We say that $q(x)$ is a \emph{strong heir extension over A} of $q\!\upharpoonright_\mathcal{M}(x)$ if   for all finite $m\subseteq \mathcal{M}$ 
 $$(\forall \varphi(x,y)\in L)(\forall b\subseteq B)[ \varphi(x,b)\in q(x) \implies (\exists b'\subseteq \mathcal{M})(\varphi(x,b')\in q(x)\wedge b\underset{Am}{\equiv}b') ].$$
	\end{defin}

Note that if $q \in S(\C)$ is a strong heir extension over $A$ of $q\!\upharpoonright_\mathcal{M}(x)$, then $\mathcal{M}$ is an $\aleph_0$-saturated model in the language $L_A$ (i.e., $L$ expanded by constants from $A$). Conversely, if $\mathcal{M}$ is an $\aleph_0$-saturated model in $L_A$ and $q(x)\in S(\mathcal{M})$, there always exists $q'(x)\in S(B)$ which is a strong heir over $A$ of $q$ (see \cite[Lemma 3.3]{doi:10.1142/S0219061319500120}). 

	\begin{lema}\label{strong heir invariance}
		Assume that $q(x)\in S(\mathcal{M})$ is $A$-invariant (for some $A\subseteq \mathcal{M}$) and $q'(x)\in S(\mathfrak{C})$ is a strong heir extension over $A$ of $q(x)$. 
 Then $q'(x)$ is the unique global $A$-invariant extension of $q(x)$.
	\end{lema}
\begin{proof}
		
To show $A$-invariance, suppose for a contradiction that there are $a,b$ and $\varphi(x,a)\in q'(x)$ with $a \underset{A}\equiv b$ and $\neg \varphi(x,b)\in q'(x)$. Then, there exist $a',b'\in \mathcal{M}$ such that $a'\underset{A}{\equiv} a$ and $b' \underset{A}{\equiv} b$ for which $\varphi(x,a') \in q(x)$ while $\neg \varphi(x,b')\in q(x)$. Then $a' \underset{A}{\equiv} b'$, which contradicts the $A$-invariance of $q(x)$.
		
		Uniqueness follows from the fact that $\mathcal{M}$ is $\aleph_0$-saturated in $L_A$.
\end{proof}

	Given a partial type (possibly with parameters) $\pi(x,y)$, we say that $\pi(x,y)$ \emph{relatively defines an equivalence relation} on a type-definable set $X$ if $\pi(\C',\C')\cap X(\C')^2$ is an equivalence relation. 
	Given a type-definable equivalence relation $E$ on a type-definable set $X$, a \emph{partial type relatively defining} $E$ is any partial type $\pi(x,y)$ such that $\pi(\C',\C')\cap X(\C')^2 = E$. 
	We say that a type-definable equivalence relation $E$ on a type-definable set $X$ is \emph{countably relatively defined} if some partial type $\pi(x,y)$ relatively defining it consists of countably many formulas, and we say that $E$ is {\em relatively type-definable over $B$} (or {\em $B$-relatively type-definable}) if it is relatively defined by a partial type over $B$.
	
	Lemma \ref{reduction to a small model} gives us a useful stability criterion when an equivalence relation on $p(\C')$ is relatively type-definable over a sufficiently saturated model.

\begin{lema}\label{reduction to a small model}
	Let $\mathcal{M}\prec \mathfrak{C}$ be $\aleph_0$-saturated in $L_A$, and $\pi(x,y)$ a partial type over $\mathcal{M}$ relatively defining an equivalence relation on $p\!\upharpoonright_{\mathcal{M}}(\mathfrak{C}')$. Then, $\pi$ relatively defines an equivalence relation on $p(\mathfrak{C}')$ with stable quotient if and only if it relatively defines an equivalence relation on $p\!\upharpoonright_{\mathcal{M}}(\mathfrak{C}')$ with stable quotient.
\end{lema}

	\begin{proof}
		Firstly, note that since $\pi$ relatively defines an equivalence relation on the set $p\!\upharpoonright_{\mathcal{M}}\!\!(\C')$,  it relatively defines an equivalence relation on $p(\C')$. Let $E$ be the equivalence relation relatively defined by $\pi(x,y)$ on $p(\C')$ and let $E'$ be the equivalence relation relatively defined by $\pi(x,y)$ on $p\!\upharpoonright_{\mathcal{M}}\!\!(\C')$.
	
	Assume first that $p(\C')/E$ is unstable. Then, there exists a $\C$-indiscernible sequence $(c_i,b_i)_{i<\omega}$ such that $c_i\in p(\C')$ for all $i<\omega$ and for all $i\neq j$ 
	$$\tp\left(\bigslant{[c_i]_{E},b_j}{\C}\right)\neq \tp\left(\bigslant{[c_j]_{E},b_i}{\C}\right).$$ 
	This implies that for all $i\neq j$ we have $$ \tp\left(\bigslant{[c_i]_{E'},b_j}{\C}\right)\neq \tp\left(\bigslant{[c_j]_{E'},b_i}{\C}\right),$$
	and so  $p\!\upharpoonright_{\mathcal{M}}\!\!(\mathfrak{C}')/E'$ is unstable.
	
	Assume now that $p\!\upharpoonright_{\mathcal{M}}\!\!(\mathfrak{C}')/E'$ is unstable. This is witnessed by an $\mathcal{M}$-indiscernible sequence $(c_i,b_i)_{i<\omega}$ such that $c_i\in p\!\upharpoonright_{\mathcal{M}}\!\!(\C')$ for all $i<\omega$ and for all $i\neq j$ $$ \tp\left(\bigslant{[c_i]_{E'},b_j}{\mathcal{M}}\right)\neq \tp\left(\bigslant{[c_j]_{E'},b_i}{\mathcal{M}}\right).$$ 
	Consider $q:=\tp\left(\bigslant{(c_i,b_i)_{i<\omega}}{\mathcal{M}}\right)$ and let $q'\in S(\C)$ be a strong heir extension over $A$ of $q$. Let $(c_i',b_i')_{i<\omega}$ be a realization of $q'$. Then, 
	\begin{enumerate}
		\item $(c_i',b_i')_{i<\omega}$ is $\C$-indiscernible;
		\item $\tp(c_i'/\C)=p(x)$ for all $i<\omega$;
		\item $\tp\left(\bigslant{[c'_i]_{E},b'_j}{\C}\right)\neq \tp\left(\bigslant{[c_j']_{E},b'_i}{\C}\right)$ for all $i \ne j$.
	\end{enumerate}
	
	$(1)$ Suppose for a contradiction that $(1)$ does not hold. Then, it is witnessed by a formula (with parameters $d$ from $\C$) of the form $\varphi(x_{i_1},y_{i_1},\dots,x_{i_n},y_{i_n},d)\wedge \neg \varphi(x_{j_1},y_{j_1},\dots,x_{j_n},y_{j_n},d)$, for some $i_1<\dots<i_n$ and $j_1<\dots<j_n$. 
	Now, using that $q'$ is a strong heir extension over $A$ of $q$, we can find $d'\subseteq \mathcal{M}$ such that $$\varphi(x_{i_1},y_{i_1},\dots,x_{i_n},y_{i_n},d')\wedge \neg \varphi(x_{j_1},y_{j_1},\dots,x_{j_n},y_{j_n},d')\in q,$$ contradicting the $\mathcal{M}$-indiscernibility of  $(c_i,b_i)_{i<\omega}$. 
	
	$(2)$ follows from the fact that $\tp(c_i'/\C)$ is a strong heir extension over $A$ of $p\!\upharpoonright_{\mathcal{M}}\!\!(x)$, which has to be $p(x)$ by Lemma \ref{strong heir invariance}. 

	$(3)$ Suppose $(3)$ fails for some $i \ne j$. As $E$ is the restriction to $p(\C)$ of the equivalence relation $E'$, we see that $\tp\left(\bigslant{[c'_i]_{E},b'_j}{\C}\right)= \tp\left(\bigslant{[c_j']_{E},b'_i}{\C}\right)$ implies $\tp\left(\bigslant{[c'_i]_{E'},b'_j}{\C}\right)= \tp\left(\bigslant{[c_j']_{E'},b'_i}{\C}\right)$. So $\tp\left(\bigslant{[c'_i]_{E'},b'_j}{\mathcal{M}}\right)= \tp\left(\bigslant{[c_j']_{E'},b'_i}{\mathcal{M}}\right)$,
and hence $\tp\left(\bigslant{[c_i]_{E'},b_j}{\mathcal{M}}\right)= \tp\left(\bigslant{[c_j]_{E'},b_i}{\mathcal{M}}\right)$ because $q \subseteq q'$ and $\pi(x,y)$ is over $\mathcal{M}$. This is a contradiction.
	
	By (1), (2), and (3),  $p(\C')/E$ is unstable.
\end{proof}

Even though at first glance the requirement that $\pi(x,y)$ relatively defines an equivalence relation on $p\!\upharpoonright_{\mathcal{M}}\!\!(\mathfrak{C}')$ might not seem very natural, the following result shows that this can always be assumed.


\begin{prop}\label{decomposition of equivalence relations}
	Let $E$ be a $B$-relatively type-definable equivalence relation on $p(\C')$, for some $B \subseteq \C$. 
Then, 
$E=\bigcap_{i\in I}E_i \cap p(\C')^2$, where $|I| \leq |B| +|x| + |T|$, and for
each $i\in I$ there is a countable $B_i\subseteq \C$ such that $E_i$ is a countably $B_i$-relatively defined equivalence relation on $p\!\upharpoonright_{B_i}\!\!(\C')$. 
Thus, $E$ is the restriction to $p(\C')$ of a $B'$-type-definable equivalence relation $F$ on $p\!\upharpoonright_{B'}\!\!(\C')$ for some $B'\subseteq \C$ with $\lvert B'\rvert \leq |B| +|x| + |T|$.

Moreover, if we start from a given partial type $\pi(x,y)$ over $B$ relatively defining $E$, 
then $B'$ and $F$ in the previous sentence can be taken so that $|B'|\leq |\pi|$ and $F$ is $B'$-type-definable on $p\!\upharpoonright_{B'}\!\!(\C')$ by $\pi(x,y)$.
\end{prop}

\begin{proof}
Fix a partial type $\pi(x,y)$ relatively defining $E$ on $p(\C')$. It clearly consists of reflexive formulas and without loss of generality it is closed under conjunction. Let $\psi_0(x)$ be any formula in $p(x)$ and $\varphi_0(x,y)$ any formula in $\pi(x,y)$. 
Then the partial type $$p(x)\wedge p(y)\wedge p(z)\wedge \pi(x,y)\wedge \pi(y,z)$$ implies 
$\varphi_0(x,z)\wedge \varphi_0(z,x)$. By compactness, there are $\varphi_1(x,y)$ in $\pi(x,y)$ and $\psi_1(x)$ in $p(x)$ such that the formula $$\psi_1(x)\wedge\psi_1(y)\wedge \psi_1(z) \wedge \varphi_1(x,y)\wedge \varphi_1(y,z)$$ implies $\varphi_0(x,z) \wedge \varphi_0(z,x).$ Proceeding by induction, we construct a partial type 
	$$\{\varphi_i(x,y): i<\omega\}$$ relatively defining an equivalence relation on $\bigcap_{i<\omega}\psi_i(\C')$. 
Let $B_{\varphi_0,\psi_0}$ be a countable set containing the parameters of all the constructed formulas $\varphi_i(x,y)$ and $\psi_i(x)$, $i < \omega$.
Then, the partial type $\{\varphi_i(x,y): i<\omega\}$ clearly relatively defines over $B_{\varphi_0,\psi_0}$ an equivalence relation on $p\!\upharpoonright_{ B_{\varphi_0,\psi_0}}\!\!(\C')$. 
Applying this process separately to every $\varphi(x,y) \in \pi(x,y)$ yields the desired family of equivalence relations.	
\end{proof}

\begin{cor}\label{existence of a small model set of parameters}
Let $E$ and $\pi(x,y)$ be as in Proposition \ref{decomposition of equivalence relations}, where $B$ is $\C$-small. Then there is $\mathcal{M}\prec \C$ containing $B$ with $\lvert \mathcal{M} \rvert \leq 2^{\lvert T \rvert + \lvert A \rvert}+ \lvert B \rvert + \lvert \pi \rvert$ which is $\aleph_0$-saturated in $L_A$ and such that $\pi(x,y)$ relatively defines an equivalence relation on $p\!\upharpoonright_{\mathcal{M}}\!\!(\C')$.
\end{cor}


The following result is a criterion for when an equivalence relation on $p(\C')$ relatively type-definable over a sufficiently saturated $\C$-small model is the finest relatively type-definable equivalence relation over a $\C$-small $B\subseteq \C$ on $p(\C')$ with stable quotient.

\begin{lema}\label{equivalence for being the finest}
	Let $\mathcal{M}$ and $\pi(x,y)$ be as in Lemma \ref{reduction to a small model}, and assume that $\mathcal{M}$ is $\mathfrak{C}$-small. Then $\pi$ relatively defines the finest relatively type-definable over a $\C$-small subset of $\C$ equivalence relation on $p(\C')$ with stable quotient if and only if it relatively defines the finest $\mathcal{M}'$-type-definable equivalence relation on $p\!\upharpoonright_{\mathcal{M}'}\!\!(\mathfrak{C}')$ with stable quotient for every $\mathcal{M}'\prec \C'$ with $\lvert \mathcal{M}'\rvert \leq 2^{\lvert T \rvert + \lvert A \rvert}+ \lvert \mathcal{M} \rvert $ that is $\aleph_0$-saturated in $L_A$ and contains $\mathcal{M}$.
\end{lema}

\begin{proof}
Let $E$ be the equivalence relation relatively defined by $\pi$ on $p(\C')$ and $E'$ be the equivalence relation relatively defined by $\pi$ on $p\!\upharpoonright_{\mathcal{M}}\!\!(\mathfrak{C}')$.

($\Leftarrow$) By Lemma \ref{reduction to a small model}, the right hand side implies that $E$ has stable quotient. Assume that there exists $E_B$, a relatively type-definable equivalence relation on $p(\C')$ over some $\C$-small set of parameters $B\subseteq \C$ such that the quotient $p(\C')/E_B$ is stable and $E_B\subsetneq E$. 
Take a presentation of $E_B$ as $\bigcap_{i\in I}E_i \cap p(\C')^2$ satisfying the conclusion of Proposition \ref{decomposition of equivalence relations}. Abusing notation, write $E_i$ for $E_i \cap p(\C')^2$. As $E_B\subsetneq E$, there exists some $i\in I$ such that  $$E\cap E_i\subsetneq E.$$ 
Since $p(\C')/E_B$ is stable and $E_B\subseteq E\cap E_i$, we have that $p(\C')/E\cap E_i$ is stable. 
Pick $B_i$ as in  Proposition \ref{decomposition of equivalence relations} and choose any $\mathcal{M}'\supseteq \mathcal{M} \cup B_i$
$\aleph_0$-saturated in $L_A$, contained in $\C$ and of size at most 
$2^{|T| + \lvert A \rvert}+ \lvert \mathcal{M}\rvert$. 
By the choice of $B_i$ and $E_i$, there is  a partial type $\delta(x,y)$ over $\mathcal{M}'$ relatively defining $E_i$ which also relatively defines an equivalence relation on $p \!\upharpoonright_{\mathcal{M}'}\!\!(\C')$. Let $\rho(x,y)$ be $\pi(x,y) \wedge \delta(x,y)$. Then $\rho(x,y)$ relatively defines an equivalence relation on $p\!\upharpoonright_{\mathcal{M}'}\!\!(\C')$ and $\bigslant{p(\C')}{ \rho(\C',\C') \cap p(\C')^2}$ is stable. Hence, applying Lemma \ref{reduction to a small model}, we obtain that the quotient  $$\bigslant{p\!\upharpoonright_{\mathcal{M}'}\!\!(\C')}{ \rho(\C',\C') \cap p\!\upharpoonright_{\mathcal{M}'}\!\!(\C')^2}.$$ is stable. Moreover, 
$$\rho(\C',\C')\cap p\!\upharpoonright_{\mathcal{M}'}\!\!(\C')^2\subsetneq\pi(\C',\C')\cap p\!\upharpoonright_{\mathcal{M}'}\!\!(\C')^2.$$
Thus, we have proved that the right hand side of the lemma fails.


($\Rightarrow$) 
By Lemma \ref{reduction to a small model}, the left hand side implies that $E'$ is stable. 
Assume that the right hand side does not hold, witnessed by a model $\mathcal{M}'$ of size at most $ 2^{\lvert T \rvert + \lvert A \rvert}+ \lvert \mathcal{M} \rvert $ that is $\aleph_0$-saturated in $L_A$ and contains $\mathcal{M}$   
 and a partial type $\rho(x,y)$ over $\mathcal{M}'$. 
 By saturation of $\C$, we can assume that $\mathcal{M}'\subseteq \C$. 
Hence,  by Lemma \ref{reduction to a small model}, the fact that the quotient $\bigslant{p\!\upharpoonright_{\mathcal{M}'}\!\!(\mathfrak{C}')}{\rho(\mathfrak{C}',\mathfrak{C}')\cap p\!\upharpoonright_{\mathcal{M}'}\!\!(\mathfrak{C}')^2}$ is stable implies that the quotient  $\bigslant{p(\mathfrak{C}')}{\rho(\mathfrak{C}',\mathfrak{C}')\cap p(\mathfrak{C}')^2}$ is stable. Let $b_1,b_2\in p\!\upharpoonright_{\mathcal{M}'}\!\!(\C')$ be elements witnessing $$\rho(\mathfrak{C}',\mathfrak{C}')\cap p\!\upharpoonright_{\mathcal{M}'}\!\!(\mathfrak{C}')^2 \subsetneq \pi(\mathfrak{C}',\mathfrak{C}')\cap p\!\upharpoonright_{\mathcal{M}'}\!\!(\mathfrak{C}')^2,$$
that is, $(b_1,b_2) \in \pi(\mathfrak{C}',\mathfrak{C}') \setminus \rho(\mathfrak{C}',\mathfrak{C}')$.
Let $q:=\tp\left(\bigslant{b_1,b_2}{\mathcal{M}'}\right)$ and let $q'\in S(\C)$ be a strong heir extension  over $A$ of $q$. 
By Lemma \ref{strong heir invariance}, any realization $(b_1',b_2')\in q'(\C')$ satisfies $b_1',b_2'\in p(\C')$, $(b_1',b_2')\in\pi(\mathfrak{C}',\mathfrak{C}')$, and $(b_1',b_2')\not \in\rho(\mathfrak{C}',\mathfrak{C}')$.
Therefore,
 $$\rho(\mathfrak{C}',\mathfrak{C}')\cap p(\mathfrak{C}')^2 \subsetneq \pi(\mathfrak{C}',\mathfrak{C}')\cap p(\mathfrak{C}')^2,$$
 which contradicts the minimality of $E$.
\end{proof}

Let $\C\prec \C_1\prec \C'$ be such that $\C_1$ is $\C'$-small and $\kappa_1$-saturated 
with $\kappa_1\geq \kappa$. A set is {\em $\C_1$-small} if its cardinality is smaller than $\kappa_1$. Let $p_1(x)\in S(\C_1)$ be the unique $A$-invariant extension of $p(x)$.

\begin{cor}\label{corollary: from C to C_1}
Assume that $E$ is  the finest relatively type-definable over a $\C$-small subset of $\C$ equivalence relation on $p(\C')$ with stable quotient. Then $E\cap p_1(\C')^2$ is the finest relatively type-definable over a $\C_1$-small subset of $\C_1$ equivalence relation on $p_1(\C')$ with stable quotient.
\end{cor}

\begin{proof}
Using Corollary 
\ref{existence of a small model set of parameters}, we can find a $\mathfrak{C}$-small  $\mathcal{M} \prec \C$ which is
$\aleph_0$-saturated in $L_A$ and a partial type $\pi(x,y)$ over $\mathcal{M}$ relatively defining $E$ and relatively defining an equivalence relation on $p\!\upharpoonright_{\mathcal{M}}\!\!(\C')$. By Lemma \ref{equivalence for being the finest}, the right hand side of the equivalence in Lemma \ref{equivalence for being the finest} holds. But this right hand side does not depend on the choice of $\C$, and so, again by Lemma \ref{equivalence for being the finest}, $E\cap p_1(\C')^2$ is the finest relatively type-definable over a $\C_1$-small subset of $\C_1$ equivalence relation on $p_1(\C')$ with stable quotient.
\end{proof}


However, there is no obvious transfer going in the opposite direction (i.e., from $\C_1$ to $\C$), as an application of 
 Corollary \ref{existence of a small model set of parameters} for $p_1$ may produce a model  $\mathcal{M} \prec \C_1$ whose cardinality is bigger than the degree of saturation of $\C$, and then we cannot embed it into $\C$ via an automorphism. We have only the following corollary.

\begin{cor}\label{corollary: from C_1 to C}
Assume that $E$ is the finest relatively type-definable over a $\C_1$-small subset of $\C_1$ equivalence relation on $p_1(\C')$ with stable quotient, and suppose that $E$ is relatively defined by a type $\pi(x,y,B)$ over a $\C$-small set $B$. 
Let $\sigma \in \aut(\C_1/A)$ be such that $\sigma[B] \subseteq \mathfrak{C}$. Then $\pi(\C',\C',\sigma[B]) \cap p(\C')^2$ is the finest relatively type-definable over a $\C$-small subset of $\C$ equivalence relation on $p(\C')$ with stable quotient.
\end{cor}

\begin{proof}
By Corollary \ref{existence of a small model set of parameters} applied to $\C_1$ and $p_1$ in place of $\C$ and $p$, there is $\mathcal{M}\prec \C_1$ containing $B$ with $\lvert \mathcal{M} \rvert \leq 2^{\lvert T \rvert + \lvert A \rvert}+ \lvert B \rvert + \lvert x \rvert$ which is $\aleph_0$-saturated in $L_A$ and such that $\pi(x,y,B)$ relatively defines an equivalence relation on $p_1\!\upharpoonright_{\mathcal{M}}\!\!(\C')$. Since $\kappa >2^{\lvert T \rvert + \lvert A \rvert}+ \lvert B \rvert + \lvert x \rvert$, we can modify $\sigma $ outside $A \cup B$ so that $\sigma[\mathcal{M}] \subseteq \C$.

By assumption and Lemma \ref{equivalence for being the finest}, the right hand side of that lemma holds for $p_1$ in place of $p$. Since $\sigma(p_1)=p_1$, it still holds for $p_1$ and $\sigma[\mathcal{M}]$ in place of $\mathcal{M}$. 
Since this right hand side does not depend on $\C_1$ and we have $\sigma[\mathcal{M}] \subseteq \C$, it holds for $p$ and $\sigma[\mathcal{M}]$, so by Lemma \ref{equivalence for being the finest}, we get that  $\pi(\C',\C',\sigma[B]) \cap p(\C')^2$ is the finest relatively type-definable over a $\C$-small subset of $\C$ equivalence relation on $p(\C')$ with stable quotient.
\end{proof}

The following proposition and its proof was proposed by the referee.

\begin{prop}
		The finest relatively type-definable over a $\C$-small subset of $\C$ equivalence relation on $p(\C')$ with stable quotient exists if and only if the finest $\C$-type-definable equivalence relation on $p(\C')$ with stable quotient is relatively type-definable over a $\C$-small subset of $\C$, and, in that case, both equivalence relations coincide.
\end{prop}

\begin{proof}
	Let $F$ be the finest $\C$-type-definable equivalence relation on $p(\C')$ with stable quotient. 
	Every relatively type-definable over a $\C$-small subset of $\C$ equivalence relation on $p(\C')$ with stable quotient is coarser than $F$. Thus, if $F$ is is relatively type-definable over a $\C$-small subset of $\C$, then it is the finest one with stable quotient. Conversely, suppose that the finest relatively type-definable over a $\C$-small subset of $\C$ equivalence relation on $p(\C')$ with stable quotient exists and denote it by $E$. As we have already pointed out, we have $F\subseteq E$. On the other hand, let $\pi(x,y)$ be a partial type over a $\C$-small subset of $\C$ relatively defining $E$ and $\rho(x,y)$ a partial type over $\C$ defining $F$. 
	Pick $\C\prec \C_1\prec \C'$ such that $\C_1$ is $\kappa_1$-saturated 
with $\kappa_1>\lvert \C\rvert$. 
By Corollary \ref{corollary: from C to C_1}, $\pi(x,y)$ relatively defines the finest relatively type-definable over a $\C_1$-small subset of $\C_1$ equivalence relation on $p_1(\C')$ with stable quotient. 
Since $p_1(\C')\subseteq p(\C')$ and $\kappa_1>\lvert \C\rvert$, we have that $\rho(x,y)$ relatively defines, over a $\C_1$-small subset of $\C_1$, an equivalence relation on $p_1(\C')$ with stable quotient. Hence, $\pi(x,y)\cup p_1(x)\cup p_1(y)\models \rho(x,y)$. 
	Consider any formula $\phi(x,y,c_0)$ implied by $\rho(x,y)$, where $c_0 \in \C$. 
	By compactness, there is a formula $\psi(x,c_1) \in p_1(x)$ and a formula $\Delta(x,y,c_2)$ implied $\pi(x,y)$ with $c_2 \in \C$ and such that $$\Delta(x,y,c_2)\wedge \psi(x,c_1)\wedge\psi(y,c_1)\models \varphi(x,y,c_0).$$ Now, take $c\in\C$ such that $\tp(c,c_0,c_2/A)=\tp(c_1,c_0,c_2/A)$. Then,  $\Delta(x,y,c_2)\wedge \psi(x,c)\wedge\psi(y,c)\models \varphi(x,y,c_0).$ On the other hand, by $A$-invariance of $p_1(x)$, we get $\psi(x,c)\in p(x)=p_1\!\upharpoonright_{\C}\!\!(x)$. Therefore, $\pi(x,y)\cup p(x)\cup p(y)\models \varphi(x,y,c_0)$. As $\varphi$ was arbitrary, we get  $\pi(x,y)\cup p(x)\cup p(y)\models \rho(x,y)$, so $E\subseteq F$, concluding $E=F$.
\end{proof}

\section{The main theorem}\label{section: main theorem}

The goal of this section is to prove the theorem stated in the introduction (see Theorem \ref{Est exists}). 

We use results on relatively type-definable subsets of the group of automorphisms of $\C'$ extracted from \cite{hrushovski2021order}. 
The following is Definition 2.14 of \cite{hrushovski2021order}, which extends the notion of relatively definable subset of the automorphism group of the monster model from \cite[Appendix A]{KPR18}.

\begin{defin}\label{relatively-type-def}
 By a  {\em relatively type-definable} subset of $\auto(\C')$, we mean a subset of the form $\{ \sigma \in \aut(\C') : \C' \models \pi(\sigma(a), b))\}$ for some partial type $\pi(x, y)$ without parameters, 
 where $x$ and $y$ are $\C'$-small 
tuples of variables, and $a$, $b$ are corresponding tuples from $\C'$.
\end{defin}

In particular, given a partial type $\pi(x,y,z)$ over the empty set, a ($\C'$-small) set of parameters $A$ and ($\C'$-small) tuples $a,b,c$ in $\C'$ corresponding to $x,y,z$, respectively, we have a relatively type-definable subset of $\auto(\C')$ of the form $$A_{\pi(x;y,z);a;b,c}(\C'/A):=\{\sigma\in \auto(\C'/A): \C'\models \pi(\sigma(a);b,c)\}.$$
In this section, if $A=\emptyset$, we will omit $(\C'/A)$, and when it is clear 
how the variables are arranged, we will denote sets of the form $A_{\pi(x;y,z);a;a,c}(\C'/A)$ as $A_{\pi;a;c}(\C'/A)$.

We use relatively type-definable sets of the group $\auto(\C')$ to prove the following:

\begin{lema}\label{IP condition}
Let $a\in \mathfrak{C}'$ and a sequence $(a_i)_{i<\omega} \subseteq \mathfrak{C}'$ (of $\C'$-small tuples $a_i$) be such that $a_0\underset{a}{\equiv}a_i$ for all $i<\omega $ and $a \models
p\!\upharpoonright_{a_{<\omega}}$.
Let $\pi(x,y,z)$ be a partial type over the empty set such that for every $i<\omega$ the partial type $\pi(x,y,a_i)$ relatively defines an equivalence relation on $p\!\upharpoonright_{a_i}\!\!(\mathfrak{C'})$. Assume that there is a formula $\varphi(x,y,z)$ implied by $\pi(x,y,z)$ such that for every $i<\omega$ 
$$ \bigcap_{j\neq i}\pi(\C',\C',a_j)\cap (p\!\upharpoonright_{a_{<\omega}}\!\!(\C'))^2 \not\subseteq \varphi(\C',\C',a_i).$$ Then $T$ has IP.
\end{lema}

To prove this result, we need the following three observations on relatively type-definable subsets of $\auto(\C')$ of special kind.

\begin{external claim}\label{stabilizer} 
 Let $a$, $(a_i)_{i<\omega}$, and  $\pi(x,y,z)$ be as in Lemma \ref{IP condition}, and let $E_{a_i}$ be the equivalence relation on $p\!\upharpoonright_{a_i}\!\!(\C')$ relatively defined by $\pi(x,y,a_i)$. Then, for all $i<\omega$, $ A_{\pi;a;a_i}(\C'/a_{i})$ is the stabilizer of the class $[a]_{E_{a_i}}$ under the action of $\auto(\C'/a_{i})$, and 
	$ A_{\pi;a;a_i}(\C'/a_{<\omega})$ is the stabilizer of the class $[a]_{E_{a_i}}$ under the action of $\auto(\C'/a_{<\omega})$.
\end{external claim}

\begin{proof}
It is clear that $\auto(\C'/a_i)$ preserves both $p\!\upharpoonright_{a_i}\!\!(\C')$ and $E_{a_i}$.

Let $\sigma\in A_{\pi;a;a_i}(\C'/a_i)$. By the definition of $A_{\pi;a;a_i}$, we have $\models \pi(\sigma(a),a,a_i)$. Hence, $\sigma(a)\in [a]_{E_{a_i}}$, and so $\sigma([a]_{E_{a_i}})= [a]_{E_{a_i}}$. Thus, we have proved that 
$$ A_{\pi;a;a_i}(\C'/a_{i})\subseteq \stab_{\auto(\C'/a_{i})}([a]_{E_{a_i}}).$$

	Conversely, let $\sigma\in\stab_{\auto(\C'/a_{i})}([a]_{E_{a_i}})$. This implies $\sigma(a)E_{a_i}a$. Hence, $\models \pi(\sigma(a),a,a_i)$, and so $\sigma\in  A_{\pi;a;a_i}(\C'/a_i)$. Thus, $$\stab_{\auto(\C'/a_{i})}([a]_{E_{a_i}}) \subseteq A_{\pi;a;a_i}(\C'/a_{i}).$$
	
	The same proof works for $\stab_{\auto(\C'/a_{<\omega})}([a]_{E_{a_i}})$.
\end{proof}	
	
	

\begin{external claim}\label{compactness}
Let  $a$, $a_0$, and $\pi(x,y,z)$ be as in Lemma \ref{IP condition}. Then, for each formula $\varphi(x,y,z)$ implied by $\pi(x,y,z)$ there is a formula $\theta(x,y,z)$ implied by $\pi(x,y,z)$ such that 
	$$  A_{\pi;a;a_0}(\C'/a_0)\cdot A_{\theta;a;a_0}(\C'/a_0)\cdot A_{\pi;a;a_0}(\C'/a_0) \subseteq A_{\varphi;a;a_0}(\C'/a_0).$$
\end{external claim}

\begin{proof}
	Let us consider the type $\pi'(x_1,x_2;y,z):=\pi(x_1,y,z)\cup \{x_2=z\}$. Then,  $$ A_{\pi;a;a_0}(\C'/a_0)=A_{\pi'(x_1,x_2;y,z); aa_0;a,a_0} .$$ Hence, by the previous claim, $A_{\pi'(x_1,x_2;y,z); aa_0;a,a_0}$ is a group, so it satisfies $$A_{\pi'(x_1,x_2;y,z); aa_0;a,a_0}^3=A_{\pi'(x_1,x_2;y,z); aa_0;a,a_0}.$$ For any formula $\varphi(x,y,z)$ implied by $\pi(x,y,z)$ we have $$A_{\pi'(x_1,x_2;y,z); aa_0;a,a_0}^3\subseteq A_{\varphi(x;y,z); a;a,a_0}.$$ 
Applying compactness (\cite[Corollary 4.8]{hrushovski2021order}), for each $\varphi(x,y,z)$ implied by $\pi(x,y,z)$ there is some $\theta(x,y,z)$ implied by $\pi(x,y,z)$ such that $$ A_{\pi'(x_1,x_2;y,z); aa_0;a,a_0} \cdot A_{\{x_2=z\}\wedge\theta(x_1;y,z); aa_0;a,a_0} \cdot A_{\pi'(x_1,x_2;y,z); aa_0;a,a_0} \subseteq  A_{\varphi(x;y,z); a;a,a_0}.$$
 Finally, since every automorphism on the left hand side belongs to $\auto(\C'/a_0)$, we conclude that $$  A_{\pi;a;a_0}(\C'/a_0)\cdot A_{\theta;a;a_0}(\C'/a_0)\cdot A_{\pi;a;a_0}(\C'/a_0) \subseteq A_{\varphi;a;a_0}(\C'/a_0).$$
\end{proof}

\begin{external claim}\label{independence of the index}
Let $a$, $(a_i)_{i<\omega}$, and $\pi(x,y,z)$ be as in Lemma \ref{IP condition}. Then, for any formulas $\varphi(x,y,z)$ and $\theta(x,y,z)$ implied by $\pi(x,y,z)$, for every $i<\omega$: $$  A_{\pi;a;a_0}(\C'/a_0)\cdot A_{\theta;a;a_0}(\C'/a_0)\cdot A_{\pi;a;a_0}(\C'/a_0) \subseteq A_{\varphi;a;a_0}(\C'/a_0).$$
	if and only if 
	$$  A_{\pi;a;a_i}(\C'/a_i)\cdot A_{\theta;a;a_i}(\C'/a_i)\cdot A_{\pi;a;a_i}(\C'/a_i) \subseteq A_{\varphi;a;a_i}(\C'/a_i).$$
\end{external claim}

\begin{proof}
	Let $\tau\in\auto(\C'/a)$ be such that $\tau(a_0)=a_i$. The conjugation by $\tau$ \begin{align*}
		\auto(\C'/a_0)&\to \auto(\C'/a_i)\\
		\sigma\hspace{20pt} &\mapsto \hspace{10pt} \tau\sigma\tau^{-1}
	\end{align*}
	is a bijection whose inverse is the conjugation by $\tau^{-1}$. Moreover, $$ \models \pi(\tau\sigma\tau^{-1}(a),a,a_i)\iff \models \pi(\sigma\tau^{-1}(a),a,a_0)\iff \models\pi(\sigma(a),a,a_0).$$
Analogous equivalences also hold for $\varphi$ and for $\theta$ in place of $\pi$. Hence, the desired equivalence follows by applying the conjugation by $\tau$.
\end{proof}

We are now ready to prove Lemma \ref{IP condition}.

\begin{proof} [{\bf Proof of Lemma \ref{IP condition}}]
	Note that for all $i<\omega$, using automorphisms of $\C'$ fixing $(a_i)_{i<\omega}$, we can reduce the condition $$ \bigcap_{j\neq i}\pi(\C',\C',a_j)\cap (p\!\upharpoonright_{a_{<\omega}}\!\!(\C'))^2\not\subseteq \varphi(\C',\C',a_i)$$ to $$ \bigcap_{j\neq i}\pi(\C',a,a_j)\cap p\!\upharpoonright_{a_{<\omega}}\!\!(\C')\not\subseteq \varphi(\C',a,a_i),$$
	because, given a pair $(c,d)$ witnessing the former condition, there exists some \\ $\sigma\in\auto(\C'/a_{<\omega})$ such that $\sigma(d)=a$, and then the pair $(\sigma(c),a)$ witnesses the latter condition. Moreover, using the same approach, one can see that the latter condition can be expressed using relatively type-definable 
subsets of $\auto(\C')$ as $$A_{\bigwedge_{j\neq i}\pi(x;y,z_j);a;a,(a_j)_{j\neq i}}(\C'/a_{<\omega}) \not\subseteq A_{\varphi(x;y,z_i);a;a,a_i}.$$
	
	For every $i<\omega$, choose some $$\sigma_i\in A_{\bigwedge_{j\neq i}\pi(x;y,z_j);a;a,(a_j)_{j\neq i}}(\C'/a_{<\omega}) \setminus A_{\varphi(x;y,z_i),a,a,a_i},$$ and let $\sigma_{I}$ denote the composition $\prod_{i\in I}\sigma_i$, for any finite $I\subseteq \omega$.

	By Claims \ref{compactness} and \ref{independence of the index}, there is a formula $\theta(x,y,z)$ implied by $\pi(x,y,z)$ such that for all $i<\omega$
	$$ A_{\pi;a;a_i}(\C'/a_i)\cdot A_{\theta;a;a_i}(\C'/a_i)\cdot A_{\pi;a;a_i}(\C'/a_i) \subseteq A_{\varphi;a;a_i}(\C'/a_i). $$ 

	\begin{claim}
		For any finite $I\subseteq \omega$ $$\models \theta(\sigma_I(a),a,a_i)\iff i\notin I.$$
	\end{claim}
	\begin{proof}[{Proof of claim}]
		Firstly, take $i\not\in I$. Then, for every $j\in I$, $\sigma_j$ belongs to the set $ A_{\pi;a;a_i}(\C'/a_{<\omega})$. By Claim \ref{stabilizer}, the set $ A_{\pi;a;a_i}(\C'/a_{<\omega})$ is a group, and so we get  $\sigma_{I}\in  A_{\pi;a;a_i}(\C'/a_{<\omega})$. Hence, $\theta(\sigma_I(a),a,a_i)$ holds.

		Now take $i\in I$ and write $I:=I_0\sqcup \{i\} \sqcup I_1$, where $I_0=\{j\in I : j<i\}$ and $I_1=\{j\in I : j>i\}$. For each $j\in I_0\cup I_1$ we have $\sigma_j\in \ A_{\pi;a;a_i}(\C'/a_{<\omega})$. Then, $\theta(\sigma_{I}(a),a,a_i)$ does not hold. Otherwise, $$ \sigma_I = \sigma_{I_0} \sigma_i \sigma_{I_1}\in  A_{\theta;a;a_i}(\C'/a_{<\omega}),$$
		which, by 
	Claim \ref{stabilizer} and	the choice of $\theta$, implies $$\sigma_i\in  A_{\varphi;a;a_i}(\C'/a_{<\omega}),$$ 
		a contradiction with our choice of $\sigma_i$. 
	\end{proof}
	The formula $\theta$ witnesses that $T$ has IP.
\end{proof}

When we write (NIP) in the statement of a result, it means that we assume that the theory $T$ has NIP.

\begin{lema}[NIP]\label{existence of an special index}
	Let $\pi(x,y,z)$ be a partial type over the empty set (with a $\C'$-small $z$), and let $a_0\subseteq \mathfrak{C}'$ be such
	 that $\pi(x,y,a_0)$ relatively defines an equivalence relation on $p\!\upharpoonright_{a_0}\!\!(\mathfrak{C}')$. Then, for any $(a_i)_{i<\lambda}$, where $\lambda \geq \beth_{(2^{(\lvert a_0\rvert +\lvert x \rvert + \lvert T\rvert + \lvert A\rvert)})^+}$ and 
	  $a_i\underset{A}{\equiv}a_0$ for all $i<\lambda$,  there exists $i<\lambda$ such that $$\bigcap_{j\neq i} \pi(\C',\C',a_j)\cap(p\!\upharpoonright_{a_{<\lambda}}\!\!(\C'))^2\subseteq\pi(\C',\C',a_i).$$
\end{lema}
\begin{proof}
	Assume the conclusion does not hold. 
	Then, for every $i<\lambda$ $$\bigcap_{j\neq i} \pi(\C',\C',a_j)\cap(p\upharpoonright_{a_{<\lambda}}\!\!(\C'))^2\not\subseteq\pi(\C',\C',a_i).$$ Take pairs $(b_i,c_i)_{i<\lambda}$ witnessing it. 
Let $(a_i',b_i',c_i')_{i<\omega} \subseteq \C'$ be an $A$-indiscernible sequence obtained by extracting indiscernibles from the sequence $(a_i,b_i,c_i)_{i<\lambda}$ (e.g. see \cite[Lemma 1.2]{doi:10.1142/S0219061303000297}).
Then, since $p$ is $A$-invariant, for all $i<\omega$ the elements $(a_i',b_i',c_i')$ satisfy:
	\begin{align*}
		(b_i',c_i')&\in\bigcap_{j\neq k}\pi(\C',\C',a_j')\cap (p\!\upharpoonright_{a'_{<\omega}}\!\!(\C'))^2;\\
		(b_i',c_i')&\not\in\pi(\C',\C',a_i');\\
		a_i'&\equiv_{A}a_0' \equiv_A a_0.
	\end{align*} 
(Note that the $A$-invariance of $p$, together with the property of being an extracted sequence, is used to ensure that $(b_i',c_i')$ belongs to $p\!\upharpoonright_{a'_{<n}}\!\!(\C')^2$ for each $n \in \omega$.) By the indiscernibility of the sequence $(a_i',b_i',c_i')_{i<\omega}$, there exists a formula $\varphi(x,y,z)$ implied by $\pi(x,y,z)$ such that for all $i<\omega$ $$(b_i',c_i')\not\in\varphi(\C',\C',a_i').$$
Take any $a \models p\!\upharpoonright_{a'_{<\omega}}$. Since $p$ is $A$-invariant, $a_i'\equiv_A a_j'$ implies  $a_i'\equiv_a a_j'$.
Moreover, since $a_i' \equiv_A a_0' \equiv_A a_0$, $\pi(x,y,a_0)$ relatively defines an equivalence relation on $p\!\upharpoonright_{a_0}\!\!(\mathfrak{C}')$, and $p$ is $A$-invariant, we get that $\pi(x,y,a_i')$ relatively defines an equivalence relation on $p\!\upharpoonright_{a_i'}\!\!(\mathfrak{C}')$ for all $i<\omega$.

Hence, the sequence $(a_i')_{i<\omega}$ together with $a$, $\pi(x,y,z)$, and $\varphi(x,y,z)$ satisfies the assumptions of Lemma \ref{IP condition}, and so we get IP, which is a contradiction.
\end{proof}



The next theorem is the main result of this paper.

\begin{teor}[NIP]\label{Est exists}
		Let $p(x)\in S_x(\mathfrak{C})$ be an $A$-invariant type with a $\C$-small $x$. Assume that the degree of saturation of $\C$ is at least $\beth_{(\beth_{2}(\lvert x \rvert+\lvert T \rvert + \lvert A\rvert))^+}$. Then, there exists a finest equivalence relation $E^{st}$ on $p(\C')$ relatively type-definable over a $\C$-small set of parameters from $\C$ and with stable quotient $p(\C')/E^{st}$. 
\end{teor}

\begin{proof}
	Let $\nu:=\beth_{(\beth_{2}(\lvert x \rvert+\lvert T \rvert + \lvert A\rvert))^+}$.

	\begin{claim}
		If for every countable partial type $\pi(x,y,z)$ over the empty set and countable tuple $a_0$ from $\mathfrak{C}$ such that $\pi(x,y,a_0)$ relatively defines an equivalence relation $E_{a_0}$ on $p(\mathfrak{C}')$ with stable quotient there is no sequence $(a_i)_{i<\nu}$ of (countable) tuples $a_i$ in $\C$ such that for all $i<\nu$ we have $a_i\underset{A}{\equiv}a_0$ and  $\bigcap_{j<i}E_{a_j}\not\subseteq E_{a_i}$, then the theorem holds.
	\end{claim}

\begin{proof}[Proof of claim]
		Consider an arbitrary collection $(E_i)_{i\in I}$ of 
		of equivalence relations on $p(\C')$ relatively type-definable over a $\C$-small subset of $\C$ and with stable quotients. Our goal is to prove that the intersection $\bigcap_{i\in I}E_i$ is a relatively type-definable over a $\C$-small subset of $\C$ equivalence relation on $p(\C')$ with stable quotient. 
		
		Using Proposition \ref{decomposition of equivalence relations}, we can write each $E_j$ as $\bigcap_{i\in I_j} F^i_j $, where each $F^i_j$ is 
		a type-definable equivalence relation on $p(\C')$ countably relatively definable over a countable subset of $\C$. Since the $F^i_j$'s are coarser than the corresponding $E_j$, each $F^i_j$ also has stable quotient. We can now write $$ \bigcap_{j\in I}E_j=\bigcap_{j\in I}\bigcap_{i\in I_j} F^i_j.$$
		Note that the number of possible countable types  
over $\emptyset$ whose instances relatively define the $F^i_j$'s is bounded by $2^{\lvert x\rvert+\lvert T\rvert}$, and the set of types over $A$ of the countable tuples of parameters used in the relative definitions of the $F^i_j$'s is bounded by $2^{\lvert T \rvert + \lvert A \rvert}$.
		 Hence, by the assumptions of the claim, the intersection $\bigcap_{j\in I}E_j$ coincides with an intersection $\bigcap_{k\in K}F^{i_k}_{j_k}$, where $\lvert K\rvert \leq 2^{\lvert T \rvert + \lvert A \rvert} \times  2^{\lvert T \rvert + \lvert x \rvert}\times \nu = \nu$. In fact, since $2^{\lvert T \rvert + |A|+ \lvert x \rvert}$ is strictly smaller than the cofinality of $\nu$, we can even get $|K| < \nu$.
		  Finally, by \cite[Remark 1.4]{MR3796277}, $\bigcap_{k\in K}F^{i_k}_{j_k}$ is a relatively type-definable over a $\C$-small subset of $\C$ (as $\C$ is $\nu$-saturated) equivalence relation on $p(\C')$ with stable quotient.
\end{proof}

Suppose the theorem fails. 
By the claim, there exists a countable type $\pi(x,y,z)$ over $\emptyset$ and a countable tuple $a_0$ in $\C$ such that $\pi(x,y,a_0)$ relatively defines an equivalence relation on $p(\C')$ with $\bigslant{p(\C')}{\pi(\C',\C',a_0)\cap p(\C')^2}$ stable and there is $(a_i)_{i<\nu}\subseteq \C$ such that for all $i<\nu$, $a_i\equiv_{A}a_0$ and $\bigcap_{j<i}\pi(\C',\C',a_j)\cap p(\C')^2\not\subseteq \pi(\C',\C',a_i)$. 
By Corollary \ref{existence of a small model set of parameters}, enlarging $a_0$, we can assume that $a_0$ enumerates an $\aleph_0$-saturated model in $L_A$ of size at most 
$2^{|T| + |A|}$ and $\pi(x,y,a_0)$ relatively defines an equivalence relation on $p\!\upharpoonright_{a_0}\!\!(\C')$; by Lemma \ref{reduction to a small model}, 
this relation also yields a stable quotient 
on $p\!\upharpoonright_{a_0}\!\!(\C')$.  

Let $(b_i,c_i)_{i<\nu}$ be a sequence witnessing that $\bigcap_{j<i}\pi(\C',\C',a_j)\cap p(\C')^2\not\subseteq \pi(\C',\C',a_i)$. 
Let $(a_i',b_i',c_i')_{i< \nu} \subseteq \C'$ be an $A$-indiscernible sequence extracted from  $(a_i,b_i,c_i)_{i<\nu}$. Then, since $p$ is $A$-invariant, we get that for all $i<\nu$
$$(b_i',c_i') \in \left(\bigcap_{j<i}\pi(\C',\C',a'_j)\cap (p\!\upharpoonright_{ a'_{<\nu}}\!\!(\C'))^2\right)\setminus \pi(\C',\C',a'_i).$$
Moreover, since $a_0' \equiv_A a_0$ 
and $p$ is $A$-invariant, we get that $\pi(x,y,a_0')$ relatively defines an equivalence relation on $p\!\upharpoonright_{a_0'}(\C')$, and we also have $a_i' \equiv_A a_0'$ for all $i<\nu$. Therefore, by Lemma \ref{existence of an special index},
there exists some $\beta <\nu$ such that  
$$(*)\;\;\;\;\;\; \bigcap_{\alpha \neq \beta} \pi(\C',\C',a_{\alpha}')\cap(p\!\upharpoonright_{a_{<\nu}'}\!\!(\C'))^2\subseteq\pi(\C',\C',a_{\beta}').$$ 

In the sequence $(a_i',b_i',c_i')_{i< \nu}$, let us insert a sequence $(d'_i,e'_i,f'_i)_{i<\omega}$ from $\C'$ in place of the element
$(a_\beta',b_{\beta}',c_{\beta}')$ so that the resulting sequence is still $A$-indiscernible. Then, since $p$ is $A$-invariant, for all $i<\omega$
$$(**)\;\;\;\;\;\; (e_i',f_i') \in \left(\bigcap_{j< i}\pi(\C',\C',d_j')\cap (p\!\upharpoonright_{a'_{\substack{\alpha<\nu\\\alpha \neq \beta}}, d'_{<\omega}}\!\!(\C'))^2\right) \setminus\pi(\C',\C',d_i').$$
Hence, due to the $A$-indiscernibility of the sequence $(d_i',e_i',f_i')_{i<\omega}$, there exists some formula $\varphi$ implied by $\pi$ such that for all $i<\omega$ we have $(e_i',f_i')\not\in\varphi(\C',\C',d_i')$. 

Moreover, since $d_i' \equiv_A a_0'$ and using the $A$-invariance of $p$, we get that $\pi(x,y,d_i')$ relatively defines an equivalence relation on $p\!\upharpoonright_{d_i'}\!\!(\C')$.

Let us consider the set 
$$X:=p\! \upharpoonright _{a'_{\substack{\alpha<\nu\\\alpha\neq \beta} }}\!\!(\C').$$ 
By the above choices and $A$-invariance of $p$, the type $\bigcup_{\substack{\alpha<\nu\\\alpha\neq \beta} }\pi(x,y,a'_\alpha)$ relatively defines an equivalence relation $E$ on $X$ with stable quotient, and the sequence $(d_i',[e_i']_E,[f_i']_E)$  is indiscernible over $$B:=A\cup \{a'_\alpha: \alpha <\nu; \alpha \neq \beta\}.$$
%
Hence, 
$$\tp\left(\bigslant{(d_j',[e_i']_E,[f_i']_E)}{B}\right)=\tp\left(\bigslant{(d_i',[e_j']_E,[f_j']_E)}{B}\right)$$
for all $j<i$. 

Let $E_i$ be the equivalence relation relatively defined by the partial type $\pi(x,y,d_i')$ on $p \!\upharpoonright _{a'_{\substack{\alpha<\nu\\\alpha\neq \beta} },d_i'}\!\!(\C')$. By $(**)$, $e_i' E_j f_i'$ for all $j<i$. Using this and the previous paragraph, we will deduce that $e_j' E_i f_j'$ for all $j<i$.

Indeed, take any $j<i$. Since $\tp\left(\bigslant{(d_j',[e_i']_E,[f_i']_E)}{B}\right)=\tp\left(\bigslant{(d_i',[e_j']_E,[f_j']_E)}{B}\right)$, there is $\sigma \in \aut(\C'/B)$ such that $\sigma(d_j',[e_i']_E,[f_i']_E)=(d_i',[e_j']_E,[f_j']_E)$. Then, by the $A$-invariance of $p$, $\sigma[E_j]=E_i$. Thus, since $e_i' E_j f_i'$, we conclude that $\sigma(e_i') E_i \sigma(f_i')$.  On the other hand, by $(*)$ and $A$-invariance of $p$, we have $E \upharpoonright_{\dom(E_i)} \subseteq E_i$, which together with the fact that $e_j' E \sigma(e_i')$, $f_j' E \sigma(f_i')$, and $e_j',f_j',\sigma(e_i'),\sigma(f_i') \in \dom(E_i)$ gives us $e_j'E_i\sigma(e_i')$ and $f_j' E_i \sigma(f_i')$. Therefore, $e_j' E_i f_j'$, as required.

We have shown that the sequence $(d_i',e_i',f_i')$ satisfies: 
\begin{align*}
		\pi(e_j',f_j',d_i') &\text{ for all } i\neq j; \hspace{3pt} i,j<\omega;\\
		\neg\varphi(e_i',f_i',d_i') &\text{ for all } i<\omega.
\end{align*}
Take any $a \models p\! \upharpoonright_{d'_{<\omega}}$. Since $d_i' \equiv_A d_0'$  for all $i<\omega$ 
and $p$ is $A$-invariant, we get that $d_i'\equiv_a d_0'$ for all $i<\omega$.
Thus, the sequence $(d_i')_{i<\omega}$ satisfies the assumption of Lemma \ref{IP condition}, and so we get IP, a contradiction.
\end{proof}

We end this section with some comments on whether the large saturation condition in Theorem \ref{Est exists} is necessary or could be eliminated. 


Note that in the above proof, in order to extract indiscernibles from the sequence $(a_i,b_i,c_i)_{i<\lambda}$, we need to know that $\nu$ is at least  $\beth_{(2^{2^{|T| + |A|} +\lvert x \rvert + |T|+|A|})^+} = \beth_{(2^{2^{|T| + |A|} +\lvert x \rvert})^+}$. On the other hand, the proof of the claim requires that any number smaller than $\nu$ is bounded in $\C$. That is why the whole proof requires that $\C$ is at least $\beth_{(2^{2^{|T| + |A|} +\lvert x \rvert})^+}$-saturated. 
In the statement of the theorem, it is enough to assume that $\C$ is $\beth_{(2^{2^{|T| + |A|} +\lvert x \rvert})^+}$-saturated; we used a bigger  degree of saturation, which is notationally more concise.

Although our proof uses essentially the assumption on the degree of saturation, one could still try to transfer the existence of the finest relatively type-definable equivalence relation from big models to their elementary substructures. 

Let $\C\prec\C_1\prec \C'$ be such that $\C_1$ is $\C'$-small and at least as saturated 
as $\C$, and let $p_1(x)\in S(\C_1)$ be the unique $A$-invariant extension of $p(x)$.

While Corollary \ref{corollary: from C to C_1} allows us to transfer the existence of the finest relatively type-definable over a $\C$-small subset of $\C$ equivalence relation on $p(\C')$ with stable quotient to the finest relatively type-definable over a $\C_1$-small subset of $\C_1$ equivalence relation on $p_1(\C')$, in order to eliminate the specific saturation assumption in Theorem \ref{Est exists}, we would need to have a transfer going in the other direction. In Corollary \ref{corollary: from C_1 to C}, we proved such a transfer but only under the additional assumption that the finest relatively type-definable  over a $\C_1$-small subset of $\C_1$ equivalence relation $E$ on $p_1(\C')$ is relatively type-definable over a $\C$-small subset. Therefore, the specific saturation assumption could be eliminated if we could answer positively the following question.

\begin{question}
	In the context of Theorem \ref{Est exists}, is $E^{st}$ always relatively type-definable over $A$?
\end{question}

In the examples studied in the next section, this turns out to be true. Also, in the context of type-definable groups studied in \cite{MR3796277}, $G^{st}$ is type-definable over the parameters over which $G$ is type-definable.

\section{Examples}\label{section examples}
We present two examples where $E^{st}$ is computed explicitly, the second example is based on \cite[Setcion 4]{10.1215/00294527-2022-0023}. 
In fact, in both examples, we give full classifications of all relatively type-definable over $\C$-small subsets of $\C$ equivalence relations on $p(\C')$, for suitably chosen $p \in S(\C)$.

\subsection*{Example 1}
Let our language be $L:=\{R_r(x,y),f_s(x): r\in\mathbb{Q}^+, s\in \mathbb{Q}\}$ and $T$ be the theory of 
$(\mathbb{R},R_r,f_s)_{r\in\mathbb{Q}^+,s\in\mathbb{Q}}$, where $f_s(x):=x+s$ and $R_r(x,y)$ holds if and only if $0\leq y-x\leq r$.

We define the directed distance between two points as a function $$d:\C'\times\C'\to \mathbb{R}\sqcup \mathbb{Q}_+\sqcup\mathbb{Q}_-\sqcup \{\infty\}$$ 
(where $\mathbb{Q}_+$ and $\mathbb{Q}_-$ are disjoint copies of $\mathbb{Q}$ which are disjoint from $\mathbb{R}\sqcup \{\infty\}$) satisfying: 
\begin{align*}
	&d(x,y)=q\in\mathbb{Q} \iff y=f_q(x);\\
	&d(x,y)=r\in \mathbb{R}^{+}\setminus\mathbb{Q} \iff \forall s_1,s_2\in \mathbb{Q}^+ \text{ such that } s_1<r<s_2, \neg R_{s_1}(x,y)\wedge R_{s_2}(x,y);\\
	&d(x,y)=q_+ \in \mathbb{Q}_+ \iff y\ne f_q(x) \text{ is  infinitely close to }f_q(x)\text{ on the right;}\\
	&d(x,y)=q_- \in  \mathbb{Q}_- \iff y \ne f_q(x) \text{ is  infinitely close to }f_q(x)\text{ on the left;}\\
	&d(x,y)=\infty \iff \neg(R_s(x,y)\lor R_s(y,x)) \text{ for all } s\in \mathbb{Q}^+.
\end{align*}
 We complete the definition of $d$ extending it symmetrically in the negative irrational case, i.e $d(y,x):=-d(x,y)$ whenever $d(x,y)\in \mathbb{R}^+ \setminus \mathbb{Q}$. 
This clearly gives us a well defined function $d$.
%
Addition on $\mathbb{R}$ is extended to  $\mathbb{R}\sqcup \mathbb{Q}_+\sqcup\mathbb{Q}_-\sqcup \{\infty\}$ in the natural way, in particular:
 \begin{itemize}
 	\item $q'+q_+:=(q'+q)_+ \text{ and }q'+q_-:=(q'+q)_-$ for any $q,q' \in \mathbb{Q}$;
 	\item $r+\infty:=\infty \text{ for any } r\in \mathbb{R}\cup \mathbb{Q}_+\cup\mathbb{Q}_-\cup \{\infty\}.$
 \end{itemize}
\begin{lema}\label{properties distance}
	Properties of the distance:
	\begin{enumerate}
		\item $d(a,f_q(b))=q+d(a,b)$ and $d(f_q(a),b)=-q+d(a,b)$;
		\item For any distinct real numbers $r_1,r_2$, if $d(a,b)=r_1$ and $d(a,c)=r_2$, then $d(b,c)=r_2-r_1$;
		\item For any irrational $r$, if $d(a,b)=r$ and $d(b,c)=0_{\pm}$, then $d(a,c)=r$;
		\item For any irrational $r$, if $d(a,b)=r=d(a,c)$, then $d(b,c)=0_\pm$.
	\end{enumerate}
\end{lema}

\begin{proof}
	$(1)$ follows from the definition of the distance.
	
	$(2)$ 
Since the rational case is covered in $(1)$, we can assume that $r_1,r_2$ are irrationals. Consider the case $0<r_1<r_2$; other cases are similar. Let $q$ be any rational bigger than $r_2-r_1$.
	 We can write $q$ as $q_2-q_1$, where $q_1,q_2$ are rationals such that $q_1<r_1<r_2<q_2$. 
Since  $R_{q'}(a,b)$ and $\neg R_{q'}(a,c)$ hold for some $q'\in \mathbb{Q}^+$ (for any $r_1<q'<r_2$) and $R_{q_1}(a,b)$ does not hold, $ R_{q_2-q_1}(b,c)$ has to hold; otherwise $R_{q_2}(a,c)$ would not hold, contradicting $d(a,c)=r_2$. Hence, $d(b,c)\leq q$

	 Let now $q$ be any positive rational smaller than $r_2-r_1$. 
	 We can write $q$ as $q_2-q_1$, where $q_1,q_2$ are rationals such that $r_1<q_1<q_2<r_2$. Since $R_{q_1}(a,b)$ holds, $R_{q_2-q_1}(b,c)$ cannot hold; otherwise, $R_{q_2}(a,c)$ would hold, contradicting $d(a,c)=r_2$. Hence, $d(b,c)\geq q$.

$(3)$ Consider the case $r>0$ and $d(b,c)=0_+$; the other cases are analogous.  Let $q$ be any rational bigger than $r$.
	 We can write $q$ as $q_1+q_2$, where  $q_1,q_2$ are rationals, $q_1>r$, and $q_2>0$. Then, $R_{q_1}(a,b)$ and $R_{q_2}(b,c)$ hold, hence so does $R_{q_1+q_2}(a,c)$. 
	 This implies that $d(a,c)\leq q$. Let now $q$ be any positive rational smaller than $r$. Then, $R_q(a,c)$ cannot hold; otherwise it would imply $R_q(a,b)$, a contradiction.

	$(4)$ 
	Consider the case $r>0$; the other case is similar. Consider any rationals $q_1$, $q_2$ satisfying $0<q_1<r<q_2$. Then, $R_{q_2-q_1}(f_{q_1}(a),b)\wedge R_{q_2-q_1}(f_{q_1}(a),c)$ holds, which imply $R_{q_2-q_1}(b,c)\vee  R_{q_2-q_1}(c,b)$. Since $q_2$ and $q_1$ were arbitrary, this means that $b$ and $c$ are infinitesimally close.
\end{proof}

It is clear that the distance determines the quantifier-free type of a pair $(a,b)$. Since our language only contains unary and binary symbols, 
the collection of distances between the elements of a given $n$-tuple determines its quantifier-free type.

\begin{prop}
	The theory $T$ has NIP and quantifier elimination.
\end{prop}
\begin{proof}
	$T$ has NIP, because it is a reduct of an o-minimal theory.
	
	We prove quantifier elimination using a back and forth argument. Let $\mathcal{M}$ and $\mathcal{N}$ be two $\aleph_0$-saturated models of $T$ and let $(a_1,\dots,a_n)$ and $(b_1,\dots,b_n)$ be tuples of elements of $\mathcal{M}$ and $\mathcal{N}$, respectively, satisfying the same quantifier free type. Choose a new element $a_{n+1}\in \mathcal{M}$. There are three cases:
	\begin{enumerate}
		\item $a_{n+1}$ is infinitely far from $a_1,\dots,a_n$;
		\item $a_{n+1}=f_q(a_i)$ for some $q\in \mathbb{Q}$ and $i=1,\dots,n$;
		\item $a_{n+1}$ is related (i.e., at finite distance) to some of the $a_i$'s but is not equal to $f_q(a_i)$ for any $q\in \mathbb{Q}$ and $i=1,\dots,n$.
	\end{enumerate}

In the first two cases, by $\aleph_0$-saturation, we can clearly choose $b_{n+1}\in\mathcal{N}$ such that $(a_1,\dots,a_{n+1})$ and $(b_1,\dots,b_{n+1})$ have the same quantifier-free type. Now, let us tackle the third case. 

In the third case, by removing the elements of the sequence $(a_1,\dots,a_n)$ which are at infinite distance from $a_{n+1}$ as well as the corresponding elements of the sequence $(b_1,\dots,b_n)$, we may assume that no $a_i$ is infinitely far from $a_{n+1}$. Note also that for each $i<n$ there is at most one $q_i\in \mathbb{Q}$ such that $f_{q_i}(a_i)$ is infinitesimally close to $a_{n+1}$. 
Let $A$ be the set of all such $f_{q_i}(a_i)$'s. 


First, consider the case when $A \ne \emptyset$. Then $A$ is a finite set totally ordered  by the relation $R_1(x,y)$ and all elements in $A$ are infinitesimally close to each other and to $a_{n+1}$. Let $B:=\{f_{q_i}(b_i) : f_{q_i}(a_i)\in A \}$. Note that all the elements in $B$ are infinitesimally close to each other and that the map sending $f_{q_i}(a_i)$ to $f_{q_i}(b_i)$ is an $R_1$-order isomorphism.
Then, by density, there exists $b_{n+1}$ with the same $R_1$-relative position  to the elements in $B$ as $ a_{n+1}$ to the corresponding elements in $A$. 
Hence, $d(b_{n+1},f_{q_i}(b_i))=d(a_{n+1},f_{q_i}(a_i))$ for each $f_{q_i}(a_i) \in A$, and, by Lemma \ref{properties distance}, this implies
$$\qftp(b_1,\dots,b_n,b_{n+1})=\qftp(a_1,\dots,a_n, a_{n+1}).$$ 

In the case when $A = \emptyset$, $d(a_i,a_{n+1})$ is irrational for every $i\leq n$. Pick $b_{n+1}$ so that $d(b_1,b_{n+1})=d(a_1,a_{n+1})$. Since $A=\emptyset$, by Lemma \ref{properties distance}(4), we get that $d(a_1,a_{n+1})\ne d(a_1,a_i)$ for all $1< i \leq n$. 
Hence, Lemma  \ref{properties distance} implies that\\
\begin{minipage}[t]{0.15\linewidth}
\hfill
\end{minipage}\hfill
\begin{minipage}[t]{0.7\linewidth}
	\begin{center}
		\vspace{1pt}
			$\qftp(b_1,\dots,b_n,b_{n+1})=\qftp(a_1,\dots,a_n, a_{n+1}).$
	\end{center}
\end{minipage}\hfill
\begin{minipage}[t]{0.15\linewidth}
			\vspace{1pt}
	\hfill\qedhere
\end{minipage}
\end{proof}

Let $p\in S_x(\C)$ be the 
$0$-invariant complete global type determined by $$ \bigwedge_{c\in\C}\bigwedge_{n\in\omega}\neg R_n(x,c)\wedge \neg R_n(c,x).$$
We denote by $E(x,y)$ the equivalence relation on $\C'$ defined by $$\bigwedge_{r\in\mathbb{Q}^+}R_r(x,y)\vee R_r(y,x)$$
and by $E\!\upharpoonright_{p}$ the equivalence relation on $p(\C')$ relatively defined by the same partial type.

\begin{lema}
	The hyperdefinable set $\C'/E(\C',\C')$ is stable.
\end{lema}
\begin{proof}
	 By \cite[Theorem 2.10]{10.1215/00294527-2022-0023}, it is enough to prove that for any $A \subseteq \C'$ with $|A| \leq \mathfrak{c}$ we have $|S_{\C'/E}(A)| \leq \mathfrak{c}$. 

 Clearly, the elements $c$ and $c'$ are in the same $E$-class if and only if $c=c'$ or $d(c,c')=0_\pm$. 
 %
Note that whenever $d(c,a) =d(c',a)\neq \infty$, then $cEc'$. Therefore, specifying the distance $d(c,a) \ne \infty$ from $c$ to a given element $a \in A$ determines the class $[c]_E$. On the other hand, by q.e., the condition saying that $d(c,a)=\infty$ for all $a \in A$ determines $\tp(c/A)$. Therefore, $|S_{\C'/E}(A)| \leq \mathfrak{c}\times\mathfrak{c} +1= \mathfrak{c}$.
\end{proof}

\begin{prop}
	The only equivalence relations on $p(\C')$ relatively type-definable over a $\C$-small subset of $\C$ are equality, $E\!\upharpoonright_p$, and the total equivalence relation. 
\end{prop}

\begin{proof}
		Let $F(x,y)$ be any equivalence relation on $p(\C')$ relatively type-definable over a $\C$-small subset of $\C$. Let $S_n(x,y):=R_n(x,y)\vee R_n(y,x)$. There are two cases.
		
		Case 1: There are $ a,b\in p(\C')$ such that  $aFb$ and $\models\bigwedge_{n\in\mathbb{N}}\neg S_n(a,b)$. For any $c,d\in p(\C')$ we can find $e\in p(\C')$ such that $\models\bigwedge_{n\in\mathbb{N}}\neg S_n(c,e) \wedge \bigwedge_{n\in\mathbb{N}}\neg S_n(d,e)$. 
		Hence, by q.e., $$(d,e)\equiv_{\C}(a,b)\equiv_{\C}(c,e).$$ As $F$ is $\C$-invariant, we conclude that $cFd$. This implies that $F$ is the total relation.
		
		Case 2: For any $ a,b\in p(\C')$ with $aFb$ there exists $n\in \mathbb{N}$ such that $\models S_n(a,b)$. 
		
		First, we show that $aFb$ implies $aE\!\upharpoonright_pb$. Assume that it is not the case. Then there exists $m\in \mathbb{Q}^+$ such that $aFb$ and $\neg S_m(a,b)$. 
On the other hand, $S_n(a,b)$ for some $n \in \mathbb{N}$.
Since $a \equiv_\C b$, there is $\sigma\in\auto(\C'/\C)$ satisfying $\sigma(a)=b$. 
Let $b_i:=\sigma^i(a)$ for $i<\omega$.
Clearly, $$(a,b)\equiv_{\C}(b,b_2)\equiv_{\C}(b_2,b_3)\equiv_{\C}\cdots . $$ 
		We deduce that for all $k\in \mathbb{N}$, $aFb_k$ and $\models \neg S_{km}(a,b_k)$. Hence, by compactness, there exists $b'\in p(\C')$ such that $aFb'$ and $\models \neg S_n(a,b')$ for all $n\in \mathbb{N}$, contradicting the hypothesis of the second case.

		Finally, if $F$ is not equality, there exist elements $a\neq b \in p(\C')$ such that $aFb$, and so $aE\!\upharpoonright_pb$ by the last paragraph. Take any distinct $c,d\in p(\C')$ satisfying $cE\!\upharpoonright_{p}d$. Then, by q.e., either $(a,b)\equiv_{\C}(c,d)$ or $(a,b)\equiv_{\C}(d,c)$. Both cases imply $cFd$, which means that $F$ and $E\!\upharpoonright_p$ are the same equivalence relation.
\end{proof}
Since $p(\C')$ is not stable, we obtain the following:
\begin{cor}
		The equivalence relation $E\!\upharpoonright_p$ is the finest equivalence relation on $p(\C')$ relatively type-definable over a $\C$-small set of parameters from $\C$ and  with stable quotient, 
that is $E^{st} = E\!\upharpoonright_p$.
\end{cor}

\subsection*{Example 2}
This example is based on \cite[Section 4]{10.1215/00294527-2022-0023}. We work in the language $L:=\{+,-,1,R_r(x,y): r\in\mathbb{Q}^+\}$ and our theory $T$ is $\Th((\mathbb{R},+,-,1,R_r(x,y))_{r\in\mathbb{Q}^+})$, where $\mathbb{R}\models R_r(x,y)$ if and only if $0\leq y-x\leq r$. 

The next result was proven in  \cite[Proposition 4.1, Proposition 4.8]{10.1215/00294527-2022-0023}.
\begin{fact}
	The theory $T$ has $NIP$ and quantifier elimination.
\end{fact}

Without loss of generality, for convenience we can assume that $\C'$ is a reduct of a monster model of $\Th(\mathbb{R},+,-,1,\leq)$. So it makes sense to use $\leq$.
Let $p\in S_x(\C)$ be the complete $0$-invariant global type determined by 
$$ \{ \neg R_r(x,c)\wedge \neg R_r(c,x) : c\in \C, r \in \mathbb{Q}^+ \}.$$

As in the previous example, let $S_r(x,y):= R_r(x,y) \vee R_r(y,x)$. We say that $x,y$ are {\em related} if $S_r(x,y)$ holds for some $r \in \mathbb{Q}^+$. 
We denote by $E(x,y)$ the equivalence relation on $p(\C')$ relatively defined by $$\bigwedge_{r\in\mathbb{Q}^+}S_r(x,y).$$
In other words, this is the relation on $p(\C')$ of lying in the same coset modulo the subgroup of all infinitesimals in $\C'$ which will be denoted by $\mu$.

Other possible relatively type-definable over a $\C$-small subset of $\C$ equivalence relations on $p(\C')$ are as follows. Take any $c \in \C$. Let $E_c$ be the equivalence relation on $p(\C')$ 
given by $xE_c y$ if and only if $x=y$ or $x+y=c$.
It is clear that this is an equivalence relation on $p(\C')$ relatively defined by a type over $c$. We also have the equivalence relation $E_c^\mu$ given by  $ x E_c^\mu y$ if and only if $xEy$ 
or $(x+y)Ec$, 
which is also relatively defined by a type over $c$.

For any non-empty $\C$-small set $A$ of positive infinitesimals in $\C$ we will consider the equivalence relation $E_A$ on $p(\C')$ given as 
$$\bigwedge_{a \in A} \bigwedge_{n \in \mathbb{N}^+} |x-y| \leq \frac{1}{n} a.$$
Note that this relation is relatively type-definable over $A$ on $p(\C')$  in the original language $L$ by the following condition
$$\bigwedge_{a \in A} \bigwedge_{n \in \mathbb{N}^+} R_1(n(x-y),a) \wedge R_1(n(y-x),a).$$

One can also combine the above examples to produce one more class of equivalence relations on $p(\C')$. 
Take any $c \in \C$ and any non-empty $\C$-small set $A$ of positive infinitesimals in $\C$. Let $\mu_A$ be the infinitesimals in $\C'$ defined by  
$$\bigwedge_{a \in A} \bigwedge_{n \in \mathbb{N}^+} |x| \leq \frac{1}{n} a.$$ 
Then we have the equivalence relation $E_{A,c}$ on $p(\C')$ given by $x E_{A,c} y$ if and only if $xE_Ay$ or $(x+y) E_A c$,
 which is clearly relatively defined on $p(\C')$ by a type over $Ac$.

\begin{teor}\label{theorem: classification}
 The only equivalence relations on $p(\C')$ relatively type-definable over a $\C$-small subset of $\C$ are: the total equivalence relation, equality, $E$, the relations of the form $E_c$ or $E_c^\mu$ (where $c \in \C$), and the relations of the form $E_A$  or $E_{A,c}$ for any non-empty $\C$-small set $A$ of positive infinitesimals in $\C$ and any $c \in \C$. 
\end{teor}

In the proof below, by a 
 non-constant term $t(x,y)$ (in the language $L$) we mean an expression $nx +my +k$, where $m,n,k \in \mathbb{Z}$ and $m \ne 0$ or $n \ne 0$.

\begin{proof}
	Let $F(x,y)$ be an arbitrary equivalence relation on $p(\C')$ relatively type-definable over a $\C$-small subset of $\C$. 
	
\begin{claim}
Either $F$ is the total equivalence relation, or $F$ is finer than $E_c^\mu$ (i.e., $F \subseteq E_c^\mu$) for some $c \in \C$ .
\end{claim}

\begin{proof}[Proof of Claim] 
We consider two cases.

	Case 1: There are $ a$, $b\in p(\C')$ such that $aFb$ and $\models \bigwedge_{n \in \mathbb{Q}^+} \bigwedge_{c\in\C}\neg S_n(t(a,b),c)$
	for all non-constant terms $t(x,y)$. Take any $a',b'\in p(\C')$. 
By compactness and $|\C|^+$-saturation of $\C'$, we can find $d'\in p(\C')$ such that $\models \bigwedge_{n \in \mathbb{Q}^+} \bigwedge_{c\in\C}\neg S_n(t(a',d'),c)$ and $\models \bigwedge_{n \in \mathbb{Q}^+} \bigwedge_{c\in\C}\neg S_n(t(b',d'),c)$ for all non-constant terms $t(x,y)$. Then, by q.e. and  
	 \cite[Remark 4.6]{10.1215/00294527-2022-0023},
	$(a',d')\equiv_{\C}(a,b)\equiv_{\C}(b',d')$. Since $F$ is $\C$-invariant, we conclude that $a'Fb'$, hence $F$ is the total equivalence relation.
	
	Case 2: For any $ a$, $b\in p(\C')$ with $aFb$ there are $n \in \mathbb{Q}^+$, $c\in \C$, and a non-constant term $t(x,y)$ such that $\models S_n(t(a,b),c)$. Suppose that for every $c \in \C$, $F$ is not finer than  $E_c^\mu$. We will reach a contradiction, but this will require quite a bit of work.

First, we claim that there are $a,b\in p(\C')$ such that 
$$(*) \;\;\;\;\;\; aFb \;\;\textrm{and}\;\; \models \bigwedge_{q\in \mathbb{Q}^+ } \neg S_q(a,b) \;\; \textrm{and}\;\; a+b\in p(\C').
$$ 

Firstly, note that by (topological) compactness of the intervals $[-r,r]$, $r \in \mathbb{Q}^+$, we easily get that $a\not\in p(\C')$ if and only if $a\in c+\mu$ for some $c\in \C$.
Assume that $(*)$ does not hold, that is, for any $aFb$ we have $aE_c^\mu b$ for some $c\in \C$ or $S_n(a,b)\wedge \neg S_m(a,b)$ for some $m,n\in \mathbb{Q}^+$. Since $F$ is not contained in any $E_c^\mu$, either we get a pair $(a,b) \in F$ such that 
$S_m(a,b) \wedge \neg S_n(a,b)$ for some $m,n\in \mathbb{Q}^+$, or we get two pairs $(a,b), (a',b') \in F$ and elements $c,c' \in \C$ such that $c-c' \notin \mu$ and $a+b -c \in \mu$ and $a'+b'-c' \in \mu$. In this second case, applying an an automorphism of $\C'$ over $\C$  mapping $a'$ to $a$, we may assume that $a'=a$, and so we get $F(b,b')$ and $b-b' \in c-c' + \mu$. Then $b+b' \in 2b' +c -c' + \mu$ is not related to any element of $\C$ (as $2b'$ is not related),  
so $b+b' \in p(\C')$. Since we assumed that $(*)$ fails, we conclude that  $S_m(b,b') \wedge \neg S_n(b,b')$ for some $m,n\in \mathbb{Q}^+$. 
In this way, the whole second case reduces to the first one, i.e. we have a pair $(a,b) \in F$ with  $S_m(a,b) \wedge \neg S_n(a,b)$ for some $m,n\in \mathbb{Q}^+$.
	
	Let $\sigma\in\auto(\C'/\C)$ be such that $\sigma(a)=b$; set $b_k:=\sigma^k(a)$. 
We produced an infinite sequence 
		\begin{center}
		\begin{tikzpicture}
			\node at (0, 0)  {$a$};
			\node at (1, 0)  {$b_1$};
			\node at (2, 0)  {$b_2$};
			\node at (3, 0)  {$\cdots$};
			\draw[-stealth]   (0.1,-0.15) to[out=-60,in=-120](0.9,-0.15);
			\node at (0.5, -0.5)  {$\sigma$};
			\draw[-stealth]   (1.1,-0.15) to[out=-60,in=-120](1.9,-0.15);
			\node at (1.5, -0.5)  {$\sigma$};
			\draw[-stealth]   (2.1,-0.15) to[out=-60,in=-120](2.9,-0.15);
			\node at (2.5, -0.5)  {$\sigma$};
		\end{tikzpicture}
	\end{center}
Then for all $k\in\mathbb{N}^+$, $aFb_k$ and $\models S_{km}(a,b_k)$ and $\models \neg S_{kn}(a,b_k)$. Since  $\models S_{km}(a,b_k)$ and $b_k$ is not related to anything in $\C$, we get $\bigwedge_{q\in \mathbb{Q}^+ } \bigwedge_{c \in \C} \neg S_q(a, -b_k + c)$, 
that is $a+b_k \in p(\C')$.
As we can use arbitrarily large $k$, by compactness (or rather $|\C|^+$-saturation of $\C'$), there exist $b$  such that $(a,b)$ satisfies $(*)$, a contradiction.
 
 We will show now that there is $b'\in p(\C')$ such that 
 $$(**)\;\;\;\;\;\; aFb' \;\; \textrm{with}\;\; a+b',a-b'\in p(\C').$$
 Namely, either $b':=b$ already satisfies it, or $a-b$ is related to some infinite $c\in \C$.  In the latter case, $a-b$ is related precisely to the elements from the set $c+\mathbb{R}+\mu$. 
 
Let $\sigma\in\auto(\C'/\C)$ be such that $\sigma(a)=b$; set $b_k:=\sigma^k(a)$. 
We have
 
 		\begin{center}
		\begin{tikzpicture}
			\node at (0, 0)  {$a$};
			\node at (1, 0)  {$b_1$};
			\node at (2, 0)  {$b_2$};
			\node at (3, 0)  {$\cdots$};
			\draw[-stealth]   (0.1,-0.15) to[out=-60,in=-120](0.9,-0.15);
			\node at (0.5, -0.5)  {$\sigma$};
			\draw[-stealth]   (1.1,-0.15) to[out=-60,in=-120](1.9,-0.15);
			\node at (1.5, -0.5)  {$\sigma$};
			\draw[-stealth]   (2.1,-0.15) to[out=-60,in=-120](2.9,-0.15);
			\node at (2.5, -0.5)  {$\sigma$};
		\end{tikzpicture}
	\end{center}
Then $aFb_k$, and one easily checks that $a-b_k$ is related precisely to the elements from the set $kc + \R+\mu$, and so $a+b_k$ is not related to anything in $\C$. Since $c$ is infinite, the sets $kc + \mathbb{R} + \mu$ are pairwise disjoint for different $k$'s, and so we find the desired $b'$ using compactness (or rather $|\C|^+$-saturation of $\C'$).
 
Since we are working in a divisible group, 
 using \cite[Remark 4.6]{10.1215/00294527-2022-0023}, we can replace the terms $t(x,y)$ in the statement of Case 2 by expressions $t_q(x,y):=nx-my$, where $q=\frac{n}{m}$ is the reduced fraction of $q$
(i.e. $\gcd(m,n)=1$ with $m>0$). In particular, note that no term of the form $t(x,y)=nx$ or $t(x,y)=my$ can occur in Case 2 hypothesis, since that would contradict $a,b\in p(\C')$.
	Notice that for each 
	$d \in p(\C')$  there exists at most one rational $q$ such that $$S_k(t_q(a,d),c)$$ holds for some $c\in \C$ and $k\in \mathbb{Q}^+$. 
For if there existed $q \ne q'\in\mathbb{Q}$ (with reduced fractions $\frac{n}{m}$ and $\frac{n'}{m'}$, respectively), $k,k' \in \mathbb{Q}^+$, and $c,c' \in \C$ such that 
$S_{k}(t_q(a,d),c)$ and $S_{k'}(t_{q'}(a,d),c')$, 
this would imply $S_{n'k+nk'}((mn'-m'n)d,nc'-n'c)$, contradicting that $d \in p(\C')$ when $mn'-m'n\neq 0$. 

We will show now that there exists $b''\in p(\C')$ such that 
$$aFb''\;\; \textrm{and} \;\; \models \bigwedge_{q \in \mathbb{Q}} \bigwedge_{n \in \mathbb{Q}^+}\bigwedge_{c\in\C}\neg S_n(t_q(a,d),c),$$ 
contradicting the assumption of Case 2.

Namely, either $d:=b'$ does the job, or there are $q \in \mathbb{Q}$, $n \in \mathbb{Q}^+$, and $c \in \C$ such that  $S_n(t_q(a,b'),c)$. By the choice of $a$ and $b'$ satisfying $(**)$, we have that $q \notin\{-1,0,1\}$.

Again, let $\sigma\in\auto(\C'/\C)$ be such that $\sigma(a)=b'$; set $b'_k:=\sigma^k(a)$. We have
 
 		\begin{center}
		\begin{tikzpicture}
			\node at (0, 0)  {$a$};
			\node at (1, 0)  {$b_1'$};
			\node at (2, 0)  {$b_2'$};
			\node at (3, 0)  {$\cdots$};
			\draw[-stealth]   (0.1,-0.15) to[out=-60,in=-120](0.9,-0.15);
			\node at (0.5, -0.5)  {$\sigma$};
			\draw[-stealth]   (1.1,-0.15) to[out=-60,in=-120](1.9,-0.15);
			\node at (1.5, -0.5)  {$\sigma$};
			\draw[-stealth]   (2.1,-0.15) to[out=-60,in=-120](2.9,-0.15);
			\node at (2.5, -0.5)  {$\sigma$};
		\end{tikzpicture}
	\end{center}
Then $aFb_k'$ for all $k \in \mathbb{N}^+$. On the other hand, applying powers of $\sigma$, we easily conclude that for every $k \in \mathbb{N}^+$, $t_{q^{k}}(a,b_k')$ is related to some element of $\C$. Hence, by an observation above, we get that for all rationals $r \ne q^k$, $t_r(a,b_k')$ is not related to anything in $\C$. Since $q \notin\{-1,0,1\}$, we know that $q,q^2,\dots$ are pairwise distinct.  So, by compactness, the desired $b''$ exists.
\end{proof}

\begin{claim}
$F \cap E$ is either equality, or $ E$, or $E_A$ for some non-empty $\C$-small set $A$ of positive infinitesimals in $\C$.
\end{claim}

\begin{proof}[Proof of Claim]

We may assume that $F \subseteq E $, and just work with $F$. Let $B$ be a $\C$-small $\dcl$-closed subset of $\C$ over which $F$ is relatively defined on $p(\C')$. Extending the notation from before the statement of Theorem \ref{theorem: classification}, for any $B' \subseteq B$ put $$E_{B'}: =\{(x,y) \in p(\C')^2: \bigwedge_{b \in B'^{+}} \bigwedge_{n \in \mathbb{N}^+} |y-x| \leq \frac{1}{n}b\},$$
where 
$B'^{+}:= \{ b \in B': 0< b \leq 1\}$. Let $A := \bigcup\{B' \subseteq B: F \subseteq E_{B'}\}$. Then 
$$F \subseteq \bigcap \{E_{B'}: B' \subseteq B \;\; \textrm{such that}\;\; F \subseteq E_{B'}\} = E_A,$$
and, as $F \subseteq E $, we have that $1 \in A$.

We will show that either $F$ is equality, or $F= E_A$. This will clearly complete the proof of the claim (note that if $A$ does not contain any positive infinitesimals, then $E_A = E$).
Suppose $F$ is not the equality. It remains to show that $F \supseteq E_A$.

Case 1: $A=B$. Pick any distinct $\alpha ,\beta \in p(\C')$ such that $\alpha F\beta $. 
Then $$\bigwedge_{a \in A^+} |\alpha-\beta| \leq a.$$ Consider any $\alpha',\beta'  \in p(\C')$ with $\alpha'E_A\beta'$. Then either $\alpha'=\beta'$ (and so $\alpha'F\beta'$), or $\bigwedge_{a \in A^+} 0<|\beta'-\alpha'|\leq a$. In the latter case, it remains to show that $\alpha \beta \equiv_A \alpha' \beta'$  or $\alpha \beta \equiv_A \beta'\alpha'$ 
(as then $\alpha'F\beta'$, since $F$ is relatively type-definable over $A$). Without loss of generality, $\beta > \alpha$ and $\beta'>\alpha'$; equivalently, $R_1(\alpha,\beta)$ and $R_1(\alpha',\beta')$ both hold. Since $\alpha \equiv_{\C} \alpha'$, we can assume that $\alpha=\alpha'$. It suffices to show that $$\{0<t-\alpha \leq a: a \in A^+\}$$ determines a complete type over $\dcl(A,\alpha)$. 
By o-minimality 
of $(\mathbb{R},+,-,1,\leq)$, this boils down to showing that there is no $b \in \dcl^*(A,\alpha)$ with $\bigwedge_{a \in A^+} \alpha<b\leq\alpha+a$, where $\dcl^*$ is computed in the language $\{+,-,1,\leq\}$. If there was such a $b$, then, by q.e. for the theory of divisible ordered abelian groups,  
it would be of the form $\gamma + q\alpha$ for some $\gamma \in A$ and $q \in \mathbb{Q}$, and we would have $\bigwedge_{a \in A^+} 0< \gamma +(q-1)\alpha \leq a$. If $q=1$, we get 
$0<\gamma\leq \frac{1}{2}\gamma<\gamma$, a contradiction. If $q \ne 1$, we get that $\alpha$ is related to an element of $A$ which contradicts the fact that $\alpha \in p(\C')$.

Case 2: $A \subsetneq B$. 
Take any $b\in B\setminus A$. 
Then, by maximality of $A$, $F\nsubseteq E_{A\cup \{b\}}$, so there is $(x,y)\in F$ such that $(x,y)\notin  E_{A\cup \{b\}}=E_A \cap E_b$; swapping $x$ and $y$ if necessary, we may assume that $y>x$. As $F\subseteq E_A$, we have that $(x,y)\notin E_b$. 
In particular, this implies that $b \in B^+$ and $y-x > \frac{1}{n}b$ for some $n\in \mathbb{N}^+$. 
Since $F\subseteq E_A$, we have that $\lvert y-x \rvert< \frac{1}{n}a$ for all $a\in A^+$ and $n\in \mathbb{N}^+$, concluding $\bigwedge_{a \in A^+} b<a$. 

Let $\sigma \in\auto(\C'/\C)$ be such that $\sigma(x)=y$; set $y_k:=\sigma^{k}(x)$. We have
 
 		\begin{center}
		\begin{tikzpicture}
			\node at (0, 0)  {$a$};
			\node at (1, 0)  {$y_1$};
			\node at (2, 0)  {$y_2$};
			\node at (3, 0)  {$\cdots$};
			\draw[-stealth]   (0.1,-0.15) to[out=-60,in=-120](0.9,-0.15);
			\node at (0.5, -0.5)  {$\sigma$};
			\draw[-stealth]   (1.1,-0.15) to[out=-60,in=-120](1.9,-0.15);
			\node at (1.5, -0.5)  {$\sigma$};
			\draw[-stealth]   (2.1,-0.15) to[out=-60,in=-120](2.9,-0.15);
			\node at (2.5, -0.5)  {$\sigma$};
		\end{tikzpicture}
	\end{center}
We easily conclude that $F(x,y_k)$ and $y_k -x > \frac{k}{n}b$ for all $k$; in particular, $y_n -x>b$. By compactness (or rather $|\C|^{+}$-saturation of $\C'$), there exist $x',y' \in p(\C')$ such that $F(x',y')$ and:
\begin{enumerate}
	\item $\bigwedge_{a \in A^+} 0<y'-x'<a$;
	\item $\bigwedge_{b \in B^+\setminus A^+} b<y'-x'$.
\end{enumerate}

We will check now that whenever $x'',y'' \in p(\C')$ satisfy (1) and (2), then $x'y' \equiv_B x''y''$.
For that, without loss of generality, we can assume that $x'=x''$. It remains to show that the partial type 
$$\pi(t/x'):= \{0<t-x'<a: a \in A^+\} \cup \{b<t-x': b \in B^+\setminus A^+ \}$$ 
determines a complete type over $\dcl(B,x')$.  
By o-minimality of $(\mathbb{R},+,-,1,\leq)$, this boils down to showing that there is no $c \in \dcl^*(B,x')$ realizing $\pi(t/x')$, where $\dcl^*$ is computed in the language $\{+,-,1,\leq\}$. If there was such a $c$, then, by q.e. for the theory of divisible ordered abelian groups,  it would be of the form $\beta + qx'$ for some $\beta \in B$ and $q \in \mathbb{Q}$, so 
$$ \bigwedge_{a\in A^+} \bigwedge_{b \in B^+\setminus A^+}  b<\beta +(q-1)x'<a.$$
If $q=1$, we get $\bigwedge_{a \in A^+} 0<\beta<a$, so $\beta \in B^+\setminus A^+$, concluding $\beta < \beta$, a contradiction. 
If $q \ne 1$, as $1\in A$, we get that $x'$ is related to an element of $B$, which contradicts the fact that $x' \in p(\C')$.

Finally, consider any $(\alpha,\beta) \in E_A$, say with $\beta >\alpha$ so $0<\beta-\alpha<\frac{1}{2}a$ for all $a\in A^+$. 
Applying $\sigma \in \auto(\C'/\C)$ mapping $y'$ to $\alpha$, we obtain $\gamma:=\sigma(x')$ such that $\gamma F\alpha$ and $\bigwedge_{b \in B^+\setminus A^+} b<\alpha-\gamma$. Since $F\subseteq E_A$, we get  $b<\alpha -\gamma < \frac{1}{2}a$ for all $b\in B^+\setminus A^+$ and $a\in A^+$. Therefore, $b<\beta - \gamma < a$ for all $b\in B^+\setminus A^+$ and $a\in A^+$, 
So, by the previous paragraph, $\gamma \alpha \equiv_B x'y' \equiv_B \gamma\beta$. As $(x',y') \in F$, we conclude that $(\alpha,\beta) \in F$, which completes the proof of the claim.
\end{proof}

By the above two claims, in order to prove the theorem, it remains to consider the case when $E\cap F\subsetneq F \subseteq E_{c_0}^\mu$ for some $c_0 \in \C$.
By the second claim, we have the following two cases.

Case 1: $E \cap F$ is the equality. We will show that then $F =E_c$ for some $c \in \C$. Consider any $a \in p(\C')$. Since $F \ne \;=$, there exists $b \ne a$ such that $aFb$. Since $F \subseteq E_{c_0}^\mu$ and $E \cap F$ is the equality, we get that such a $b$ is unique: if $b' \ne a$ also satisfies $aFb'$, then $b,b' \in -a+c_0 + \mu$, so $b-b' \in \mu$, hence $b=b'$ because $bFb'$. 
This unique $b$ belongs to $\dcl(\C, a)$, so $g:=a+b-c_0 \in \dcl(\C,a) \cap \mu$.
Since $a$ is not related to any element of $\C$ and $\dcl$ is given by ``terms" with rational coefficients (which follows from q.e. for $T$),  we get that $g \in \C$.
Hence, $c:=a+b=c_0+g \in \C$. Applying automorphisms over $\C$, we get $F=E_c$.

Case 2: $E \cap F=E$ or $E \cap F=E_A$ for some non-empty $\C$-small set $A$ of positive infinitesimals. Since $E=E_{\{1\}}$ (with the obvious extension of the definition of $E_A$), we can write $E \cap F=E_A$, where $A$ is either a non-empty $\C$-small set $A$ of positive infinitesimals or $A=\{1\}$. We will show that then $F =E_{A,c}$ for some $c \in \C$, where $E_{\{1\},c}:=E^\mu_c$. Extend the definition of $\mu_A$ via $\mu_{\{1\}}:=\mu$.

Consider any $a \in p(\C')$. Since $E\cap F \ne F$ and $F \subseteq E_{c_0}^\mu$, there exists $b\in p(\C')$ such that $(a,b) \in F \setminus E$ and $a+b = c_0 +g$ for some $g \in \mu$. As  $E \cap F=E_A$, we get that $\sigma(g)-g \in \mu_A$ for every $\sigma \in \auto(\C'/\C a)$. 

Since $\sigma(g)-g \in \mu_A$ for every $\sigma \in \auto(\C'/\C a)$, we conclude by o-minimality of $(\mathbb{R},+,-,1,\leq)$ that, for every $\alpha\in A$ and $\in \mathbb{N}^+$, there are $c_{\alpha,n}, d_{\alpha,n} \in \dcl^*(\C,a)$ such that $g-\frac{1}{n}\alpha<c_{\alpha,n}\leq g\leq d_{\alpha,n}<g+ \frac{1}{n}\alpha$, 
where  $\dcl^*$ is the definable closure computed in the language $\{+,-,1,\leq\}$ (which coincides with $\dcl$ as both closures are given by ``terms" with rational coefficients). 
Since $a$ is not related to any element of $\C$ and for every $\alpha\in A$ and $n\in \mathbb{N}^+$ the elements $c_{\alpha,n}, d_{\alpha,n}$ are related to zero, using that $\dcl^*$ is given by  ``terms'' with rationals coefficients, we conclude that $c_{\alpha,n}$ and $d_{\alpha,n}$ belong to $\C$ for every $\alpha\in A$ and $n\in \mathbb{N}^+$. Since $A$ is $\C$-small, the set of all $c_{\alpha,n}$ and $d_{\alpha,n}$ is $\C$-small, and hence there is $e \in \C$ with $g-\frac{1}{n}\alpha< e <g + \frac{1}{n}\alpha$ for all $\alpha \in A$ and $n \in \mathbb{N}^+$. Then, $g\in e+\mu_A$ with $e\in \C$, concluding $a+b\in c+\mu_A$, where $c=c_0+e\in \C$, so $aE_{A,c}b$.

 From the conclusion of the previous paragraph and the fact that $E \cap F=E_A$, we obtain $(a+\mu_A) \cup (-a+ c +\mu_A) \subseteq [a]_F$. By automorphisms over $\C$, the same is true for any other element of $p(\C')$ in place of $a$, so $E_{A,c} \subseteq F$. The opposite inclusion easily follows using the assumptions $F \subseteq E_{c_0}^\mu$ and $E \cap F=E_A$. Namely, using automorphisms over $\C$, it is enough to show that $[a]_F \subseteq [a]_{E_{A,c}}$. Consider any $b' \in [a]_F$. Since $(a,b) \in F \setminus E$ and  $F \subseteq E_{c_0}^\mu$, we have that $b' \in a + \mu$ or $b' \in b + \mu$. As $E \cap F=E_A$, we conclude that $b' \in a + \mu_A$ (and so $b'E_{A,c}a$) or $b' \in b+ \mu_A$ (and so $b'E_{A,c}b$ which together with $aE_{A,c}b$ implies $b'E_{A,c}a$).
\end{proof}

\begin{cor}
	The equivalence relation $E$ is the finest equivalence relation on $p(\C')$ relatively type-definable over a $\C$-small set of parameters of $\C$ and with stable quotient, that is $E^{st}=E$
\end{cor}
\begin{proof}
	The quotient $p(\C')/E$ is stable by \cite[Proposition 4.9]{10.1215/00294527-2022-0023}. 
Let $F$ be a relatively type-definable over a $\C$-small subset of $\C$ equivalence relation on $p(\C')$ strictly finer than $E$. 
By Theorem \ref{theorem: classification}, $F=E_A$ for some non-empty $\C$-small set $A$ of positive infinitesimals in $\C$. 
(There is also the case when $F$ is the equality, but then $p(\C')/F=p(\C')$ is clearly unstable.)

Pick any $\alpha \in A$.
It is easy to check that for every $a,d,e\in \C'$ and infinitesimal $c\in \C'$ bigger than all infinitesimals in $\C$ and such that $|d - a|\leq \frac{1}{2}\alpha$ and $|e -(a+c)|\leq\frac{1}{2}\alpha$, we have $d<e$, and so $\neg R_1(e,d)$. And note that ``$|x-y|\leq\frac{1}{2}\alpha$'' can be written as an $L_\alpha$-formula.

	 
Take any $a \in p(\C')$ and infinitesimal $c\in \C'$ bigger than all infinitesimals in $\C$.
	 Using Ramsey's theorem and compactness, 
	 we find a $\C$-indiscernible sequence $(a_i')_{i<\omega}$ having the same Erenfeucht-Mostowski type as the sequence $(a+kc)_{k<\omega}$.
	 Then, the sequence $([a'_i]_F)_{i<\omega}$ is $\C$-indiscernible but not totally $\C$-indiscernible, since the formula $R_1(x,y)$ witnesses that $$ \tp([a'_i]_F,[a'_{i+1}]_F/ \C)\neq\tp([a'_{i+1}]_F, [a'_i]_F/ \C). $$ 
Thus, $p(\C')/F$ is unstable.
\end{proof}


\section*{Acknowledgments}

We would like to thank the anonymous referee for very careful reading and many suggestions which improved the presentation.

\printbibliography
\nocite{*}
\end{document}